\documentclass[english, a4paper ,DIV=12, parskip=full, twoside]{scrarticle}
\usepackage{enumerate}

\usepackage{local-preamble}
\shortauthor{Zimmermann, R., Bergmann, R.}
\shorttitle{Multivariate Hermite interpolation of manifold-valued data}

\author{
Ralf Zimmermann%
\thanks{Department of Mathematics and Computer Science,%
University of Southern Denmark, Odense, Denmark
(\email{zimmermann@imadasdu.dk}, \url{https://portal.findresearcher.sdu.dk/en/persons/zimmermann} \orcid{0000-0003-1692-3996});%
}
\and
Ronny Bergmann%
\thanks{Norwegian University of Science and Technology, Department of Mathematical Sciences, Trondheim, Norway (\email{ronny.bergmannn@ntnu.no}, \url{https://www.ntnu.edu/employees/ronny.bergmann}, \orcid{0000-0001-8342-7218}).%
}
}
\title{Multivariate Hermite interpolation of manifold-valued data}

\begin{document}
\maketitle

\begin{abstract}
In this paper, we propose two methods for multivariate Hermite interpolation of mani\-fold-valued functions. 
On the one hand, we approach the problem via computing suitable weighted Riemannian barycenters.
To satisfy the conditions for Hermite interpolation, the sampled derivative information is converted into a condition on the derivatives of the associated weight functions. It turns out that this requires the solution of linear systems of equations, but no vector transport is necessary.
This approach treats all given sample data points equally and is intrinsic in the sense that it does not depend on local coordinates or embeddings.\\
As an alternative, we consider Hermite interpolation in a tangent space. This is a straightforward approach, where one designated point, for example one of the sample points or (one of) their center(s) of mass, is chosen to act as the base point at which the tangent space is attached. The remaining sampled locations and sampled derivatives are mapped to said tangent space.
This requires a vector transport between different tangent spaces.
The actual interpolation is then conducted via classical vector space operations. The interpolant depends on the selected base point.\\
The validity and performance of both approaches is illustrated by means of numerical examples.
\end{abstract}   

\begin{AMS}
\href{https://mathscinet.ams.org/msc/msc2010.html?t=65D05}{65D05}
\href{https://mathscinet.ams.org/msc/msc2010.html?t=65D15}{65D15}
\href{https://mathscinet.ams.org/msc/msc2010.html?t=49Q99}{49Q99}
\href{https://mathscinet.ams.org/msc/msc2010.html?t=41A29}{41A29}
\href{https://mathscinet.ams.org/msc/msc2010.html?t=53B50}{53B50}
\end{AMS}

\section{Introduction}

In this paper, we address multivariate Hermite interpolation of a function that takes values on a complete Riemannian manifold $\mcM$
with tangent bundle $T\mcM$. More precisely, let $D\subset \R^d$ be a parameter domain and consider a differentiable function 
\begin{equation*}
 f\colon  D \to \mcM, \quad \omega\mapsto f(\omega).
\end{equation*}
Consider a sample data set consisting of $k$ parameter locations $\omega_1,\ldots, \omega_k\in D$ with corresponding function values (manifold locations) 
and partial derivatives (tangent vectors)
\begin{equation*}
 p_j = f(\omega_j)\in\mcM,
\quad 
	v_j^i\coloneqq \partial_i f(\omega_j) = \deriv\Big\vert_{t=0} f(\omega_j + te_i)
	= \diff f(\omega_j)[e_i] \in T_{f(\omega_j)}\mcM,
\end{equation*}
where $e_1,\ldots,e_d$ denotes the orthonormal basis of unit vectors in $\R^d$.
The Hermite manifold interpolation problem is formalized as follows:
\\
Find a differentiable, manifold-valued function $\hat f\colon D\to \mcM$ such that
\begin{subequations}
    \begin{align}
        \hat f(\omega_j) &= p_j\in \mcM,  \qquad j =1,\ldots,k
        \label{eq:basicHermiteMnf1}
    \\
    \partial_i \hat f(\omega_j) &= v_j^i\in T_{p_j}\mcM,  \qquad j=1,\ldots,k; \hspace{0.2cm}  i=1,\ldots,d.
       \label{eq:basicHermiteMnf2}
    \end{align}
\end{subequations}
\subsection{Original contributions}
We develop two approaches to tackle the multivariate manifold Hermite interpolation problem \eqref{eq:basicHermiteMnf1}, \eqref{eq:basicHermiteMnf2}.
The first one is via computing weighted Riemannian barycenters.
This can be considered as an extension of the interpolation approach  developed in \cite{Grohs:2015} and in \cite{Sander:2012, Sander:2016}
by incorporating derivative data. We refer to this as barycentric Hermite interpolation \textbf{(BHI)}.
The BHI method has the following main features:
\begin{enumerate}[(i)]
 \item The approach works on arbitrary Riemannian manifolds, i.e., no special structure
 (Lie Group, homogeneous space, symmetric space,...) is required.
 In order to conduct practical computations, only an algorithm for evaluating the Riemannian logarithm map must be available.
 All occurrences of Riemannian exponentials can be replaced with retractions and analogously any Riemannian logarithms by inverse retractions \cite[Section 4.1]{AbsilMahonySepulchre:2008}.
For clarity of presentation we just write exponential and logarithmic map, respectively, during this paper.
 \item The differentiability order of the interpolant is the same as that of the weight functions. Hence, when working with smooth weight functions, a smooth interpolant is produced.
 \item Although our theoretical derivation involves covariant derivatives and the Hessian forms of scalar fields on manifolds, the practical implementation of BHI does not require computing any such operators.
 \item The method works only locally, on a domain, where the Riemannian center of mass exists and is unique. The number of sample points must exceed the dimension of the manifold.
\end{enumerate}
%
The second approach is a straightforward translation of Hermite interpolation in Euclidean vector spaces to a selected tangent space of the manifold under consideration. 
We refer to this approach as tangent space Hermite interpolation \textbf{(THI)}.
The THI method has the following main features:
\begin{enumerate}[(i)]
 \item The approach works on arbitrary Riemannian manifolds, i.e., no special structure
 (Lie Group, homogeneous space, symmetric space,...) is required.
 In order to conduct practical computations, algorithms for evaluating the Riemannian exponential map and the Riemannian logarithm map must be available (or must be consistently replaced with invertible retractions).
 The differential of the Riemannian logarithm must be evaluated (or approximated). 
 \item The differentiablility order of the interpolant is the same as that of the weight functions.
 \item The method works only locally, on a star-shaped domain around a manifold location $p_c\in \mcM$, where the Riemannian normal coordinates are well-defined.
 The interpolant is constructed in the tangent space $T_{p_c}\mcM$ and thus depends on the choice of the center point $p_c$.
\end{enumerate}
We illustrate the validity and the performance of both approaches by means of numerical examples.
\subsection{Related work}
In the research literature, there are two different tracks of manifold interpolation research.
On the one hand, there is the problem of interpolating scalar functions with manifold-valued \emph{inputs} $f\colon\mcM \to \R$.
This line of research is followed, e.g., in \cite{Allasia:2018, Narcowich:1995}, but is not considered here.
On the other hand, interpolating parametric functions with manifold-valued \emph{outputs} $f\colon\R^d \to \mcM$ is investigated.
The work at hand subordinates to this setting.
Confusion may be caused by the fact that for both types of problems, one speaks of “interpolation \emph{on} a manifold”.\\
To the best of our knowledge, \emph{multivariate} Hermite interpolation for output data 
on general manifolds has not yet been considered in the research literature.
Univariate Hermite interpolation has been considered explicitly in \cite{Jakubiak:2006}
for data on compact, connected Lie groups with a bi-invariant metric.
A general approach to Hermite curve interpolation is featured in \cite{Zimmermann:2020}.\\
A related line of research is the generalization of Bézier curves and the De Casteljau-algorithm
to Riemannian manifolds, see \cite{BergmannGousenbourger:2018,GouseMassartAbsil:2018, Polthier:2013,  Noakes:2007, SAMIR:2019}.
The transition of this technique to manifolds is via replacing the inherent straight lines with geodesics.
The start and end velocities of the resulting spline are proportional to the velocity vectors of 
the geodesics that connect the first two and the last two control points, respectively (\cite[Theorem 1]{Noakes:2007}).
Hence, the method can be adapted to work for univariate Hermite interpolation.
More general subdivision schemes based on similar geodesic averages have been investigated in \cite{Dyn:2017, Wallner:2005}.
The preprint \cite{BenZion:2022} addresses univariate Hermite interpolation on the sphere via subdivision schemes.
Two-variate manifold interpolation via Bézier surfaces with $C^1$-connection of local patches has been investigated in \cite{AbsGouseWirth:2016}.\\
A multivariate Hermite-type method that is specifically tailored to interpolation problems on the Grassmann manifold is sketched in \cite[\S 3.7.4]{Amsallem:2010}.\\
Interpolation with Riemannian barycenters has been proposed in the context of geo\-desic finite elements in 
\cite{Grohs:2015} and in \cite{Sander:2012, Sander:2016}.\\
An overview over manifold interpolation methods with a special focus on model reduction applications is given in \cite{Zimmermann:2021}.
\subsection{Notational conventions}
Throughout, we assume that $\mcM$ is a complete Riemannian manifold of dimension $\dim(\mcM)= m$.
The tangent space of $\mcM$ at $p\in \mcM$ is $T_p\mcM$ and the tangent bundle is
$T\mcM= \cup_{p\in\mcM}T_p\mcM$.
As a Riemannian manifold, $\mcM$ carries a family of inner products 
$\left\{\langle \cdot,\cdot \rangle_{p}\colon T_p\mcM\times T_p\mcM\to \R, (v,w)\mapsto \langle v,w \rangle_{p}\right\}_{p\in\mcM}$
that depends smoothly on the base point $p\in\mcM$.\\
The differential of a function $f\colon \R^d\to \R^m$ at a point $p$ is denoted by 
$\Diff{}f(p)\colon \R^d \to \R^m$. For a two-arguments function $L\colon\R^m\times \R^d\to \R^m, (q,\omega) \mapsto L(q,\omega)$,
the linear maps
$\Diff_q L(q,\omega)\colon \R^m\to \R^m$ and $\Diff_\omega L(q,\omega)\colon \R^d\to \R^m$ denote the differentials of
$q\mapsto L(q,\omega)$ (with $\omega$ considered as fixed) and
$\omega \mapsto L(q,\omega)$ (with $q$ considered as fixed), respectively. 
The differential at a point $p$ of a function $f\colon\mcM \to \mcN$ between differentiable manifolds $\mcM,\mcN$ is denoted by
$\diff{}f(p)\colon T_p\mcM \to T_{f(p)}\mcN$.
For $L\colon\mcM\times \R^d\to \mcM, (q, \omega) \mapsto L(q,\omega)$, the notations $\diff{}_qL(q,\omega)$ and $\diff{}_\omega L(q,\omega)$ are to be understood in analogy to the above. 
\\
A vector field is a smooth function $X\colon\mcM\to T\mcM$ such that $X(p)\in T_p\mcM$ for all $p\in\mcM$.
The set of smooth vector fields on $\mcM$ is denoted by $\mathcal{X}(\mcM)$.
We use the symbol $\nabla$ to denote the (unique) Levi-Civita connection on $\mcM$. For $X,Y\in\mathcal{X}(\mcM)$, the notation $\nabla_XY\in \mathcal{X}(\mcM)$ denotes the covariant derivative of the vector field $Y$ in the direction of the vector field $X$. A vector field along a curve $c\colon I\to \mcM$ is a mapping 
$X_c\colon I \to T\mcM$ such that $X_c(t) \in T_{c(t)}\mcM$. The covariant derivative of $X_c$ along $c$ is $\covDeriv[t]X_c$.
If there is an ambient vector field $X\in\mathcal{X}(\mcM)$ such that $X(c(t)) = X_c(t)$, then $\covDeriv[t]X_c = \nabla_{\dot c(t)} X$.
For a scalar function $g\colon\mcM\to \R$, the gradient vector field $\grad g\in \mathcal{X}(\mcM)$ is point-wise defined by
$\langle \grad g(p), v\rangle_p = \diff{}g(p)[v]$ for all $v\in T_p\mcM$.
The covariant derivative of the gradient field yields the Hessian
$\Hess g [X] = \nabla_X \grad g$
and gives rise to the endomorphism
$T_p\mcM \ni v\mapsto \Hess g(p)[v] = (\nabla_v \grad g)(p)\in T_p\mcM$.\\
For a scalar two-parameter function $L\colon\mcM\times \R^d, (q,\omega)\mapsto L(q,\omega)\in\R$, we will write $\grad_q L(p,\omega)$ for the gradient by $q$ evaluated at $(p,\omega)$, which is defined by
$
\langle \grad_q L(p,\omega), v\rangle_p = \diff{}_q L(p,\omega)[v]$, for all $v\in T_p\mcM$.
Likewise, $\Hess_q L(p,\omega)[v]$ denotes the Hessian of $L(q,\omega)$ by $q$ evaluated at $(p,\omega)$ applied to the tangent vector $v\in T_p\mcM$.
\subsection{Organization of the paper}
Section \ref{sec:RiemannBary} recaps the barycentric interpolation methods. In Section \ref{sec:baryHermite}, the approach is extended to the Hermite setting.
The alternative approach of tangent space Hermite interpolation is outlined in Section \ref{sec:THI}. 
Numerical experiments are featured in Section \ref{sec:numex}, and Section \ref{sec:conclusions} concludes the paper.
%
\section{Interpolation with weighted Riemannian barycenters}
\label{sec:RiemannBary}
The Riemannian barycenter or Riemannian center of mass\footnote{Here, we introduce Riemannian barycenter for discrete data sets; for centers w.r.t.\ a general mass distribution, see Karcher's original paper \cite{Karcher:1977}, Section 1.} or Fr{\'e}chet mean of a sample data set
$\{p_1,\ldots,p_k\}\subset \mcM$ on a manifold
is defined as the minimizer of the Riemannian objective function
\begin{equation*}
 \mcM\ni q \mapsto L(q) = \frac{1}{2} \sum_{j=1}^k w_j \dist(q, p_j)^2,
\end{equation*}
where $\dist(q,p_j)$ is the Riemannian distance between the manifold locations $q$ and $p_j$ and $w_j\geq 0$ are scalar weights such that $\sum_{j=1}^k w_j = 1$.
Formally, the latter requirements mean that the center of mass is taken with respect to a discrete positive measure of unit weight.
This definition generalizes the notion of the barycentric mean in Euclidean spaces, cf. \Cref{app:EuclideanCase} of the supplements.
However, on curved manifolds, the global center might not be unique. Moreover,
local minimizers may appear. For more details, see \cite{Karcher:1977} and \cite{AfsariTronVidal:2013}, which also give uniqueness criteria.

\subsection{Interpolation via optimization}
Interpolation can be performed by computing weighted Riemannian centers.
More precisely, let $f\colon\R^d\supset D \to \mcM$ and
let $\omega_1,\ldots,\omega_k\subset D$ be a set of parameter locations and let 
$p_j=f(\omega_j)\in \mcM$, $j=1,\ldots,k$ be the corresponding sampled manifold locations on $\mcM$.
The interpolant is then defined on the convex hull $\text{conv}\{\omega_1,\ldots,\omega_k\}\subset D\subset \R^d$ of the samples.\\
Let $\{\varphi_j:\omega\mapsto \varphi_j(\omega)\in \R\mid j=1,\ldots,k\}$ be a suitable set of multivariate, scalar-valued interpolation weight functions with $\varphi_l(\omega_j) = \delta_{lj}$ and $ \sum_{j=1}^k \varphi_j(\omega) \equiv 1$.
Such weight functions can be constructed, e.g., as Lagrangians, \cite{Sander:2016},  or radial basis functions, \cite{Buhmann:2003}.
The interpolant $q^* \approx f(\omega^*)\in \mcM$ at an unsampled parameter 
location $\omega^*\in \text{conv}\{\omega_1,\ldots,\omega_k\}$ can be taken to be the minimizer
\begin{equation}
 \label{eq:baryInterp}
 \omega^* \mapsto f(\omega^*) = q^* \coloneqq \argmin_{q\in \mcM} L(q, \omega^*),
 \quad \text{where } L(q, \omega) \coloneqq \frac12 \sum_{j=1}^k \varphi_j(\omega) \dist(q,p_j)^2.
\end{equation}
Since the weight functions may attain negative values, this corresponds to taking the center of mass with respect to a discrete {\em signed} measure of unit weight, see the discussion in \cite[Section 3]{Sander:2016}.
Computing $q^*$ thus requires one to solve a Riemannian optimization problem.
At sample location $\omega_l$, one has indeed that
\begin{equation*}
2L(q, \omega_l) = 
\sum_{j=1}^k \varphi_j(\omega_l) \dist(q,p_j)^2
=\sum_{j=1}^k \delta_{lj} \dist(q,p_j)^2= \dist(q,p_l)^2,
\end{equation*}
which has the unique global minimum at $q^* = p_l$.
Hence, the function $\hat f\colon \omega \mapsto \argmin_{q\in \mcM} L(q, \omega)$ satisfies the basic interpolation conditions.
\begin{remark}
Under certain conditions that locally ensure the existence and uniqueness of Riemannian barycenters,
the minima of \eqref{eq:baryInterp} are exactly at the zeros of the associated gradient.
Moreover, they depend smoothly on the parameters $(q, \omega)$ if the weight functions $\varphi_j$ are smooth, 
see \cite[Theorems 3.19 \& 4.1]{Sander:2016}.
As a consequence, the interpolant is smooth under these conditions.
\end{remark}
\subsection{The gradient and Hessian of the Riemannian distance function}
Interpolation via barycenters is an optimization task.
As a rule, numerical optimization requires the computation of the gradient of the objective function. 
The next theorem, due to Karcher, provides the gradient of the squared  
Riemannian distance function.
\begin{theorem}[{\cite[Thm. 1.2]{Karcher:1977}}]
\label{thm:grad_dist}
Let $\mcM$ be a complete Riemannian manifold and let $p\in \mcM$.
Let $B_\rho(p)$ be a geodesic ball of radius $\rho$ around $p$ such that the geodesics between any two points inside $B_\rho(p)$ are unique and minimizing.
Define 
\begin{equation*}
	L_p\colon B_\rho(p)\to \R, \quad q \mapsto \frac12 \dist(q,p)^2.
\end{equation*}
Then
\begin{equation*}
	\grad L_p(q) = -\Log_q(p),
\end{equation*}
where $\Log_q =  (\Exp_q)^{-1}$ is the Riemannian logarithm map.
\end{theorem}
Because Karcher works with barycenters with respect to mass distributions and also to make this exposition self-contained, we recap the proof in \Cref{app:proofs}.\\
Eventually, we also need information on the Hessian of $L_p(q) = \frac12 \dist(q,p)^2$ at $q=p$. 
The Riemannian Hesse form of a scalar function $g\colon\mcM\to \R$ at $q$ is 
\begin{equation}
\label{eq:Hesse_def}
   \Hess g(q)\colon T_q\mcM\to T_q\mcM, \quad v\mapsto \Hess g(q)[v] = (\nabla_v \grad g)(q),
\end{equation}
see \cite[\S 6A]{Kuhnel:2015}.
\begin{remark} Actually, the Riemannian Hesse $(1,1)$-tensor maps vector fields to vector fields,
\begin{equation*}
  \Hess g\colon \mathcal{X}(\mcM) \to \mathcal{X}(\mcM),
  \quad X \mapsto \nabla_X \grad g.
\end{equation*}
Yet, since a tensor is a point-wise object, $(\nabla_v \grad g)(q) = (\nabla_X \grad g)(q)$ for all vector fields $X\in\mathcal{X}(\mcM)$ with $X(q) =v$.
Hence, it makes sense to consider the Hessian at $q$ as an endomorphism of $T_q\mcM$.
Additional background information on the Riemannian Hessian is given in  \Cref{app:proofs} of the supplements.
\end{remark}
\begin{theorem}
\label{thm:Hesse_id}
Consider the setting of Theorem \ref{thm:grad_dist}.
For $p$ fixed, the Hesse form of the function $q\mapsto L_p(q) = \frac12 \dist(q,p)^2$ at $p$ is the identity,
\begin{equation*}
    \Hess L_p(p)= \id_{T_p\mcM}\colon T_p\mcM \to T_p\mcM.
\end{equation*}
\end{theorem}
\begin{proof}
We use Karcher's approach of computing the Hessian 
via a variation through geodesics, see \Cref{thm:Karcher1977} 
and \Cref{fig:GeoVar} in the supplement.
This shows that for a geodesic $t\mapsto \gamma(t)$ with $\gamma(0) = q$, $\dot \gamma(0) = v$
and the associated variation of geodesics
\begin{equation*}
	(s,t)\mapsto c_p(s,t) = \Exp_p(s\Log_p(\gamma(t))),
\end{equation*}
it holds
\begin{equation*}
 \langle \Hess L_p(q)[v]\ ,\ v\rangle =
	 \langle\covDeriv[s] \partial_t c_p(1,0)\ ,\ v \rangle.
\end{equation*}
In the special case, where the geodesic $\gamma$ starts from $q=p$ with velocity $v\in T_p\mcM$,
we obtain
\begin{align*}
    \partial_t\big\vert_{t=0} c_p(s,t) &= 
    \diff\left(\Exp_p\right)_{s\Log_p(\gamma(0))}
  \left(s \diff\left(\Log_p\right)_{\gamma(0)}(\dot \gamma (0))\right)
  \\
  &= \diff\left(\Exp_p\right)_{s\cdot 0}
  \left(s \diff\left(\Log_p\right)_{p}(v)\right)
  = sv \in T_p\mcM,
\end{align*}
because $\diff{}(\Exp_p)_0=\id_{T_p\mcM} = \diff{}(\Log_p)_p$.
Note that $s\mapsto \partial_t c_p(s,0)\in T_p\mcM$ is a vector field along the point curve
$\alpha\colon s\mapsto \alpha(s)\equiv p$.
Therefore, the covariant derivative coincides with the usual derivative and we obtain
\begin{equation*}
   \covDeriv[s] \partial_t c_p(s,0) = \deriv[s](sv) = v.
\end{equation*}
As a consequence,
\begin{equation*}
 \langle \Hess L_p(p)[v]\hspace{0.1cm},\hspace{0.1cm} v\rangle =
	 \langle v, v \rangle \quad \text{ for all } v\in T_p\mcM.
\end{equation*}
The Hessian is symmetric. Via polarization, it is uniquely determined by terms of the above form. This yields $\Hess L_p(p)[v] = v$.
\end{proof}

By \Cref{thm:grad_dist}, the gradient of the objective function $L$ in \eqref{eq:baryInterp} by $q$ is
\begin{equation}
 \label{eq:baryGrad}
 \grad_q L(q, \omega) = \sum_{j=1}^k \varphi_j(\omega)\cdot (-\Log_q(p_j)) \in T_q\mcM.
\end{equation}

At a sample location $p_l$, the gradient (by $q$) of the objective function $L(q,\omega)$ in the barycentric interpolation problem \eqref{eq:baryInterp} vanishes, because of $\Log_{p_l}(p_l) = 0\in T_{p_l}\mcM$. 
By \Cref{thm:Hesse_id}, the Hessian (again with respect to the $q$-argument) is
\begin{equation}
    \label{eq:Hesse_id_at_p}
    \Hess_{q} L(p_l, \omega_l)[v] = \sum_{j=1}^k \varphi_j(\omega_l) \Hess L_{p_j} (p_l)[v]
    =  \Hess L_{p_l} (p_l)[v] = v \in T_{p_l}\mcM,
\end{equation}
because $\varphi_j(\omega_l)=\delta_{lj}$.
Thus, the Hessian by $q$ at $(p_l, \omega_l)$ is the identity on $T_{p_l}\mcM$. In particular, it has full rank.
A generic gradient descent algorithm\footnote{
	There exist several different ways of computing the weighted Riemannian barycenter, see for example
	\url{https://juliamanifolds.github.io/Manifolds.jl/latest/features/statistics.html} for an overview.
} to compute the barycentric interpolant
for a function $f\colon\R^d \ni \omega \mapsto f(\omega) \in \mcM$ is given in \Cref{alg:BaryInt}

\begin{algorithm}{Interpolation via the weighted Riemannian barycenter.}
\begin{algorithmic}[1]
  \REQUIRE{Sample data set $\{p_1=f(\omega_1),\ldots, p_k=f(\omega_k)\}\subset \mcM$,
  unsampled parameter location $\omega^*\in \text{conv}(\omega_1,\ldots, \omega_{k})\subset \R^d$,
  initial guess $q_0$, convergence threshold $\tau$}
  \STATE{$k\coloneqq0$}
  \STATE{Compute $\grad_q L(q_k, \omega^*)$ according to \eqref{eq:baryGrad}}
  \WHILE{$\lVert\grad_q L(q_k, \omega^*)\rVert_{q}> \tau$}
  \STATE{select a step size $\alpha_k$}
  \STATE{$q_{k+1} \coloneqq \Exp^{\mcM}_{q_k}\left(-\alpha_k \grad_q L(q_k, \omega^*)\right)$}
  \STATE{$k\coloneqq k+1$}
  \ENDWHILE
  \ENSURE{$\hat f(\omega^*)\coloneqq q^*\coloneqq q_k\in \mcM$ interpolant of $f(\omega^*)$.}
\end{algorithmic}
\label{alg:BaryInt}
\end{algorithm}
%
%
%
%
\section{Barycentric Hermite interpolation}
\label{sec:baryHermite}
In this section, we enhance the me\-thod of weighted barycentric interpolation by including derivative information. The task is to construct an interpolant  of the form
\begin{equation*}
	\hat f\colon \omega \mapsto 	\hat f(\omega)= \argmin_{q\in \mcM} \frac12 \sum_{j=1}^k \varphi_j(\omega) \dist(q,p_j)^2 = \argmin_{q\in \mcM} L(q,\omega)
\end{equation*}
that satisfies the interpolation conditions \eqref{eq:basicHermiteMnf1}, \eqref{eq:basicHermiteMnf2}.
The requirement to meet the sampled derivatives $v_j^i=\partial_i f(\omega_j)$ entails conditions on the partial derivatives of the weight functions $\varphi_j(\omega)$.
We work under the general assumption that the interpolation procedure takes place on a domain, where the weighted Riemannian barycenters exist and are unique.
For a detailed analysis of interpolation via Riemannian barycenters, we refer to \cite{Sander:2012, Sander:2016} and \cite{Grohs:2015}.
\subsection{Tracking the barycenters via the implicit function theorem}
\label{sec:TechPrep}
Introduce the parametric gradient field
\begin{equation}
\label{eq:objectiveG_VF}
  G\colon\mcM \times \R^d \to T\mcM, \quad (q, \omega) \mapsto G(q,\omega) \coloneqq \grad_q L(q,\omega) = -\sum_{j=1}^k \varphi_j(\omega) \Log_q(p_j)
\end{equation}
and note that $G$ is a smooth vector field on the product manifold $\mcM \times \R^d$.
Suppose that $G$ vanishes at $(q^*, \omega^*)$.
Our strategy is to parameterize the zero-sets locally via the implicit function theorem.
By differentiating the corresponding implicit function, we will establish a relation between the derivatives of the interpolation weight functions $\omega\mapsto \varphi_i(\omega)$ of \eqref{eq:baryInterp} and the sampled derivatives \eqref{eq:basicHermiteMnf2}.\\
The implicit function theorem has a close relative, the inverse function theorem,
and both theorems rely on a full-rank condition for a certain differential.
While a manifold version of the inverse function theorem appears in many standard textbooks on differential geometry
(e.g., \cite[Thm 4.5]{Lee:2012}), we were not able to locate a textbook reference for a manifold counterpart to the implicit function theorem.
Yet, because both $\mcM$ and $T\mcM$ are differentiable manifolds, it is straightforward to transfer the classical implicit function theorem,
and often, it is simply taken for granted (as  in \cite[Thm. 2.2]{Sander:2012}, \cite[Thm 4.1.]{Sander:2016}). The interesting part is how the rank condition transforms. This is the contents of the following lemma.
\begin{lemma}
\label{lem:Hessefullrank4implicit}
Let $\mcM$ be a Riemannian manifold.
\begin{enumerate}
\item Let $G\colon\mcM \to T\mcM$ be a smooth vector field. 
At locations, where $G$ vanishes, the differential of $G$ coincides with the covariant derivative, i.e., 
at $q^*\in \mcM$ with $G(q^*) = 0$, it holds
\begin{equation*}
    \diff{}G(q^*)[v] = (\nabla_v G) (q^*), \quad\text{ for all } v\in T_{q^*}\mcM.
\end{equation*}
\item  In the special case, where $G$ is the gradient vector field $\grad g\colon\mcM \to T\mcM$ of a scalar function $g\colon\mcM\to \R$,
the Hesse form of $g$ coincides with the differential of $\grad g$ at locations, where the gradient vanishes, i.e., 
at $q^*\in \mcM$ with $\grad g(q^*) = 0$, it holds
\begin{equation*}
    \diff{}(\grad g)(q^*)[v] = \Hess g(q^*) [v],  \quad\text{ for all }v\in T_{q^*}\mcM.
\end{equation*}
In particular, the differential $\diff{}(\grad g)(q^*)$ has full rank, if the Hesse form $\Hess g(q^*)$ has full rank.
\end{enumerate}
\end{lemma}
\begin{proof}
Let $v\in T_{q^*}\mcM$. We compute $\diff{}G(q^*)[v]$. 
Let $\gamma\colon I\to \mcM$ be a smooth curve with $\gamma(0) = q^*, \dot \gamma (0) = v$.
On a suitably small neighborhood $\tilde M$ around $q^*$, construct a local, orthonormal 
frame of vector fields $\{E_i\in \mathcal{X}(\mcM)\mid i=1,\ldots, m\}$
as outlined in \cite[Exercise 3.2, p. 24]{Lee:1997}.
This means that we obtain vector fields $E_i :\tilde M \to T\mcM$, $ i=1,\ldots, m$ such that at each $p\in \tilde M$,
$\{E_i(p)\mid i=1,\ldots, m\}$ is an orthonormal basis of $T_p\mcM$, i.e., 
\begin{equation*}
	 \langle E_i (p), E_j(p) \rangle_{p} = \delta_{ij}\quad \text{ for all } p\in \tilde M.
\end{equation*}
W.l.o.g., assume that the image of $\gamma$ is contained in $\tilde M$.
We express the vector field $G$ along $\gamma$ in terms of the orthonormal frame
\begin{equation*}
 (G \circ \gamma) (t) =
	    \sum_i \langle (G \circ \gamma) (t), (E_i\circ \gamma)(t) \rangle_{\gamma(t)} E_i(\gamma(t)),
\end{equation*}
and differentiate according to the product rule  (\cite[Chapter 2, Prop. 3.2]{DoCarmo:1992})
\begin{align*}
 \diff{}G(q^*)[v] &= \deriv[t]\big\vert_{t=0} (G \circ \gamma) (t)\\
	&= \sum_i \Bigl[ \deriv[t]\big\vert_{t=0}\left( \langle (G \circ \gamma) (t), E_i(\gamma(t)) \rangle_{\gamma(t)}\right) E_i(q^*)\\
	& \hspace{0.5cm}  + \langle 
            G (q^*)
    , E_i(q^*) \rangle_{q^*}  \deriv[t]\big\vert_{t=0} (E_i\circ \gamma)(t)\Bigr]\\
	 &= \sum_i \left( \langle \frac{\mathrm{D}(G \circ \gamma)}{\mathrm{d}t}(0), E_i(q^*) \rangle_{q^*} + 
	                           \langle
                            G (q^*)
                            , \frac{\mathrm{D}(E_i\circ \gamma)}{\mathrm{d}t}(0) \rangle_{q^*} 
	                  \right) E_i(q^*)\\
  &=  \sum_i  \langle (\nabla_v G) (q^*), E_i(q^*) \rangle_{q^*}  E_i(q^*) =  (\nabla_v G) (q^*).
\end{align*}
Note that we used twice that $G(q^*) = 0$.
If $G$ is the gradient vector field of a scalar function $G = \grad g$, then the latter becomes the Hessian,
 $(\nabla_v \grad g) (q^*) = \Hess g(q^*) [v]$.
\end{proof}
By Lemma \ref{lem:Hessefullrank4implicit}, the condition on the rank of the differential 
of a vector field translates to a condition on the covariant derivative (at points, where the vector field vanishes). This makes it straightforward to transfer any of the classical proofs of the implicit function theorem to the manifold setting.
\begin{theorem}[An implicit function theorem for vector fields]
\label{thm:ImplFunVF}
Let $\mcM$ be an $m$-di\-men\-sio\-nal Rie\-mannian manifold and let $G\colon\mcM\times \R^d \to T\mcM, (q,\omega)\mapsto G(q,\omega)\in T_q\mcM$ be a smooth mapping such that both
$q\mapsto G(q, \omega)\in T_q\mcM$ with $\omega$ considered as fixed and 
$\omega \mapsto G(q,\omega)\in T_q\mcM$ with $q$ considered as fixed are smooth vector fields.
Let $(q^*, \omega^*)\in \mcM\times \R^d$ be a point such that $G (q^*, \omega^*) = 0\in T_{q^*}\mcM$
and let $q\mapsto G^*(q) \coloneqq G(q,\omega^*)$.
Assume that the linear form $(\nabla_{(\cdot)} G^*)(q^*)\colon T_{q^*}\mcM \to T_{q^*}\mcM, v\mapsto (\nabla_{v} G^*)(q^*) $ has full rank $m$.\\
%
%
%
Then there exists an open domain $\Omega\subset\R^d$ around $\omega^*$ and a differentiable function $\hat f\colon\Omega \to \mcM, \omega \mapsto \hat f(\omega)$
such that 
\begin{equation}
\label{eq:ImplFunVF}
	 \Omega\ni \omega \mapsto 	G(\hat f(\omega), \omega))\equiv 0\in T_{\hat f(\omega)}\mcM.
\end{equation}
\end{theorem}
\begin{proof}
Use the standard trick to obtain the implicit function theorem from the inverse function theorem
see, e.g.,  \cite[Thm. C.40]{Lee:2012}. To this end, consider the auxiliary function
\[
    \gamma\colon\mcM \times \R^d \to T\mcM \times \R^d. \quad (q,\omega) \mapsto 
        \begin{pmatrix}
    \Gamma_1(q,\omega)\\
    \Gamma_2(q, \omega)
    \end{pmatrix}=
    \begin{pmatrix}
    G(q,\omega)\\
    \omega
    \end{pmatrix}.
\]
Note that $\Gamma$ is a smooth map between differentiable product manifolds.
By Lemma \ref{lem:Hessefullrank4implicit}, the differential of the first component function of $\Gamma$ by $q$ at the point $(q^*,\omega^*)$ is given by
$\diff{}_qG(q^*, \omega^*)[v] = (\nabla_{v} G^*)(q^*)$ and thus has full rank by assumption.
The differential of the second component function by $q$ vanishes, while its differential by $\omega$ is the identity on $\R^d$. As a consequence, the differential of $\Gamma$ at $(q^*,\omega^*)$
has full rank and the inverse function theorem yields the existence of a local inverse $\Gamma^{-1}= H = (H_1, H_2)$.
It holds 
\[
    \begin{pmatrix}
    v\\
    \omega
    \end{pmatrix}
    =
    \Gamma(H(v, \omega)))
    =
        \begin{pmatrix}
    G\left((H_1(v,\omega), H_2(v,\omega)\right)\\
    H_2(v,\omega)
    \end{pmatrix}.
  \]
 In particular, $H_2(v,\omega) = \omega$ and the function $\hat f\colon\omega\mapsto H_1(0, \omega)$ is the sought-after implicit function with
 $G(\hat f(\omega), \omega)  = G(H_1(0, \omega), H_2(v,\omega)) \equiv 0$ on a local domain, where $\hat f$ is defined.
\end{proof}
An alternative proof that works with local coordinates is given in \cite[Theorem 3.1]{Seguin:2022}. However, this proof omits the technical detail stated in Lemma \ref{lem:Hessefullrank4implicit}.
%
%

%
In the setting of Theorem \ref{thm:ImplFunVF}, implicit differentiation of \eqref{eq:ImplFunVF} along a curve 
$c:I\to \Omega$ with $c(0) = \omega, \dot c(0) = w$ yields
\begin{align*}
0&=\deriv[t]\big\vert_{t=0} G(\hat f( c(t)), c(t)) 
  = \left(\diff_qG((\hat f(\omega),\omega)), \diff{}_\omega G((\hat f(\omega),\omega))\right) 
     \begin{bmatrix}
            \diff{}\hat f_{\omega} [w]\\
            w
          \end{bmatrix}\\
 &= \diff_qG(\hat f(\omega),\omega)\left[\diff{}\hat f(\omega) [w]\right]
  +  \diff_\omega G(\hat f(\omega),\omega)[w]
\end{align*}
%
Since $G(\hat f(\omega),\omega) = 0$, Lemma \ref{lem:Hessefullrank4implicit} applies and yields
\begin{equation}
    \label{eq:ImplDiff_eq}
    \left(\nabla_{d\hat f(\omega) [w]} G(\cdot,\omega)\right) (\hat f(\omega)) = -  \diff{}_\omega G(\hat f(\omega),\omega)[w] \in T_{\hat f(\omega)}\mcM.
\end{equation}
Similar considerations were made in \cite[Section 6]{Sander:2016}, but for a different purpose.
\subsection{A relation between the derivatives of the interpolation weight functions and the sampled derivatives}
Now, we apply the results of Subsection \ref{sec:TechPrep} to the specific gradient field $G$ of \eqref{eq:objectiveG_VF}.
Recall that $G$ is the gradient vector field (with respect to $q$) of the scalar function $L(q,\omega) = \frac12 \sum_{j=1}^k \varphi_j(\omega)\dist(q,p_j)^2$
so that $G(q,\omega) \coloneqq \grad_q L(q,\omega) = -\sum_{j=1}^k \varphi_j(\omega) \Log_q(p_j)$.
By construction, $G$ vanishes at the sample locations $q=p_l$, $\omega = \omega_l$, $j=1,\ldots, k$.
Moreover, by Theorem \ref{thm:Hesse_id} and \eqref{eq:Hesse_id_at_p}, the Hessians at the sample locations have full rank.
Theorem \ref{thm:ImplFunVF} yields locally an implicit parameterization of the zero set of the gradient field around each sample point.
Since the sought-after interpolant is determined via an optimization problem, and thus, in turn, via the zero set of the gradient field,
the interpolant coincides with the implicit function $\hat f$. In particular, $\hat f(\omega_l) = p_l$.\\
Fix a sample point $(\omega_l,p_l)$, $l\in 1,\ldots k$.
Let $\{e_i\in\R^d\mid i=1,\ldots,d\}$ denote the Cartesian unit vectors in $\R^d$.
With $w=e_i$, $\omega=\omega_l$, eq. \eqref{eq:ImplDiff_eq} relates the partial derivatives $v^i_l \coloneqq\partial_i \hat f(\omega_l)=\diff{}\hat f(\omega_l)[e_i]$ of $\hat f$ to the partial derivatives of the interpolation weight functions,
\begin{align*}
  \left(\nabla_{v^i_l} \grad_q L(p_l,\omega_l) \right) (p_l)
    &= -  \diff{}_\omega G(p_l,\omega_l)[e_i] \quad \Leftrightarrow\\
  \Hess_{q} L(p_l,\omega_l)[v^i_l]
  &= v^i_l = \sum_{j=1}^k \partial_i\varphi_j(\omega_l) \Log_{p_l}(p_j)
  = \sum_{j=1, j\neq l}^k \partial_i\varphi_j(\omega_l) \Log_{p_l}(p_j).
\end{align*}
This is in direct correspondence with \eqref{eq:implderivEuc} in the Euclidean setting, see \Cref{app:EuclideanCase} in the supplements.

In order to be able to match any prescribed, sampled partial derivatives 
$v^i_l = \partial_i \hat f(\omega_l)$, we require the sets
\begin{equation*}
	\mathcal L_l \coloneqq \{\Log_{p_l}(p_j)\mid  j=1,\ldots,k, j\neq l\} 
	\subset T_{p_l}\mcM, \quad l=1,\ldots k
\end{equation*}
to span the full tangent space $T_{p_l}\mcM$.
Then any choice of interpolation weight function $\varphi_j$ with $\varphi_j(\omega_l)=\delta_{jl}$ that satisfies
\begin{equation}
\label{eq:HermiteConditionSystem}
 v^i_l = \sum_{j=1, j\neq l}^k \partial_i\varphi_j(\omega_l) \Log_{p_l}(p_j)
\end{equation}
yields an interpolant $\hat f$ that satisfies the interpolation conditions \eqref{eq:basicHermiteMnf1}, \eqref{eq:basicHermiteMnf2}. 
\begin{remark}
Recall that $\text{dim}(\mcM)=m$.
For every fixed index $l\in \{1,\ldots,k\}$, let $\{E^l_j\mid  j=1,\ldots,m\}\subset T_{p_l}\mcM$ be a local orthonormal frame. Then, each of the tangent vectors
$\Log_{p_l}(p_j)\in T_{p_l}\mcM$ features a representation with $m$ coordinate coefficients
\[
    \Log_{p_l}(p_j) = x^j_l(1) E^l_1 + \cdots + x^j_l(m) E^l_m.
\]
Likewise for the sampled derivatives,
\[
    v_l^i = \alpha_l^i(1) E^l_1 + \cdots +  \alpha_l^i(m) E^l_m.
\]
With respect to the coordinates of the local frame, the unknowns in \eqref{eq:HermiteConditionSystem}, i.e., the scalar coefficients $\partial_i\varphi_j(\omega_l)$,  are determined by
\begin{equation}
    \label{eq:coeffs_in_local_frame}
  \begin{pmatrix}
    x^1_l(1) &\ldots &x^{l-1}_l(1) &x^{l+1}_l(1)&\ldots & x^k_l(1)\\
    \vdots&       &\vdots    &\vdots   &       &\vdots\\
    x^1_l(m) &\ldots &x^{l-1}_l(m) &x^{l+1}_l(m)&\ldots & x^k_l(m)
  \end{pmatrix}
  \begin{pmatrix}
    \partial_i\varphi_1(\omega_l)\\
    \vdots\\
    \partial_i\varphi_{l-1}(\omega_l)\\
    \partial_i\varphi_{l+1}(\omega_l)\\
    \vdots\\
    \partial_i\varphi_k(\omega_l)\\
  \end{pmatrix}
  =
    \begin{pmatrix}
    \alpha_l^i(1)\\
    \vdots\\
    \vdots\\
    \alpha_l^i(m)\\
  \end{pmatrix}.
\end{equation}
The system matrix in \eqref{eq:coeffs_in_local_frame} is of dimensions
$m\times (k-1)$.
In order to ensure that solutions exist, the number of sample points must be larger  than the manifold dimension, $k> m$, and the logs $\Log_{p_l}(p_j)$, $j\neq k$ must span the tangent space $T_{p_l}\mcM$.
If $k> m+1$, we have more coefficients $\partial_i\varphi_j(\omega_l)$ than are needed for a representation of
$v^i_l$ in terms of the logs. The system is underdetermined.
The additional degrees of freedom may be used to  impose extra conditions, e.g. that the derivative coefficients sum up to zero.
In fact, this is required for consistency, because $\sum_j \varphi_j(\omega)\equiv 1$, i.e., the weights correspond to a discrete signed measure of unit weight.
\end{remark}
\section{Multivariate Hermite interpolation in the tangent space}
\label{sec:THI}
As an alternative to BHI, multivariate Hermite interpolation on a manifold may also be conducted in a straightforward manner by moving all data,  i.e., sample points and derivatives, to a selected tangent space and to perform the interpolation therein.
The following steps detail this approach.
\begin{enumerate}
    \item Choose a manifold location $q_0\in \mcM$ that is to act as the `data center'. For example, $q_0$ may be one of the sample points or it may be the Riemannian barycenter of the given sample data set.
    \item Map all sampled manifold locations to the tangent space at $q_0$, i.e., compute
    $w_j = \Log_{q_0}(pj)\in T_{q_0}\mcM$, $j=1,\ldots,k$.
    \item Transport the sampled velocity vectors to the selected tangent space $T_{q_0}\mcM$.
    To this end, compute for $i=1,\ldots, d, \hspace{0.1cm} l = 1,\ldots, k$, 
    \begin{eqnarray}
    \nonumber
       T_{q_0}\mcM\ni  \hat v_l^i 
       &=& d(\Log_{q_0})_{p_l}[\partial_i f(\omega_l)]\\
       \label{eq:THI_FD_prep}
       &\approx&\frac{\Log_{q_0}( \Exp_{p_l}(\Delta t \partial_i f(\omega_l))) - \Log_{q_0}(\Exp_{p_l}(-\Delta t \partial_i f(\omega_l)))}{2\Delta t}.
    \end{eqnarray}
    The latter finite-differences approximation can be used in cases, where the Riemannian log map is not available in closed form,
    for details and remarks on the numerical accuracy, see \cite{Zimmermann:2020}.
    \item With all data gathered on one and the same tangent space, classical Hermite interpolation in vector spaces can be pursued. The interpolant is constructed as a weighted linear combination of the sampled data
    as
    \begin{equation}
    \label{eq:tanginterp}
        \hat f_{\text{tan}}(\omega) = \sum_{j=1}^k \varphi_j(\omega) w_j + \sum_{i=1}^d\sum_{l=1}^k \psi_{i,l}(\omega) \hat v_l^i \in T_{q_0}\mcM.
    \end{equation}
    \item The sampled data provides Hermite sample values for the weight functions $\varphi_j(\omega), \psi_{i,l}(\omega)$.
    Let $\vec{\varphi_j}:= \left(\varphi_j(\omega_1),\ldots,\varphi_j(\omega_k)\right)$ be the vector 
    of sampled values of the $j$th coefficient function $\omega \mapsto \varphi_j(\omega)$.
    Likewise, let
    $\vec{\psi}_{i,l}:= \left(\psi_{i,l}(\omega_1),\ldots,\psi_{i,l}(\omega_k)\right)$.
    For each of the coefficient functions, one needs to solve a multiple-input scalar-output
    Hermite interpolation problem, where the sample locations are $\omega_1,\ldots,\omega_k$
    and the sample values for the various coefficient functions are 
    \begin{equation}
        \label{eq:THI_sample_data}
        \begin{pmatrix}
        \vec{\varphi}_{j}\\
        \partial_1 \vec{\varphi}_{j}\\
        \vdots\\
        \partial_d \vec{\varphi}_{j}
        \end{pmatrix}
        =
        \begin{pmatrix}
        e_j\\
        \mathbf{0}\\
        \vdots\\
        \mathbf{0}
        \end{pmatrix},
    \quad 
     \begin{pmatrix}
        \vec{\psi}_{1,l}\\
        \partial_1 \vec{\psi}_{1,l}\\
        \vdots\\
        \partial_d \vec{\psi}_{1,l}
        \end{pmatrix}
    =
        \begin{pmatrix}
        \mathbf{0}\\
        e_l\\
        \vdots\\
        \mathbf{0}
        \end{pmatrix},
         \ldots,  
     \begin{pmatrix}
        \vec{\psi}_{d,l}\\
        \partial_1 \vec{\psi}_{d,l}\\
        \vdots\\
        \partial_d \vec{\psi}_{d,l}
        \end{pmatrix}
    =
        \begin{pmatrix}
        \mathbf{0}\\
        \mathbf{0}\\
        \vdots\\
        e_l
        \end{pmatrix} \in \R^{k(d+1)}.
   \end{equation}
   
    Here, the data is written in vector packs of size $k$ 
    ($k$ values from $\varphi_j(\omega)$ evaluated at the $k$ sample locations, 
    $k$ values from $\partial_1\varphi_j(\omega)$ evaluated at the $k$ sample locations, etc. )
    and $e_j=(0,\ldots,\stackrel{j}{1},\ldots, 0)^{\mathrm{T}}\in\R^k$ denotes the $j$th canonical unit vector.
    \item The interpolants of the coefficient functions can now be obtained from any multivariate interpolation scheme that is able to tackle Hermite data, e.g., the gradient-enhanced Kriging method, \cite[Section 7.2]{ForresterSobesterKeane:2008}.
    In summary, this requires to solve $k(d+1)$ Euclidean interpolation problems.
    The resulting tangent vector of $\eqref{eq:tanginterp}$ is mapped back to the manifold via
    \begin{equation}
        \label{eq:THI_interpolant}
        \hat f(\omega) = \Exp_{q_0}(\hat f_{\text{tan}}(\omega)) \in \mcM.
    \end{equation}
\end{enumerate}
\paragraph{Preliminary comparison.}
For simplicity, the above interpolation approach is referred to as the tangent space Hermite interpolation \textbf{(THI)} method.
The BHI method outlined in Section \ref{sec:baryHermite} works with a number of $k$ weight coefficient functions with properly adjusted partial derivatives, which are obtained by solving the linear equation systems \eqref{eq:coeffs_in_local_frame}.
The THI method works with $k(d+1)$ coefficient functions, each of which comes with its own Hermite sample data set. 
At the preprocessing stage of the BHI approach, the Riemannian logarithm function needs to be queried. In contrast, prior to conducting THI, one needs to query the Riemannian logarithm as well as its total derivative (or an approximation thereof).
In the online stage of BHI, the values of the coefficient functions at an untried location $\omega^*\in\R^d$ are determined via a Riemannian optimization problem.
In fact, it is the exact same Riemannian optimization problem that needs to be tackled for performing non-Hermite barycentric interpolation, except that the weight functions are properly adjusted a priori. The practitioner may choose any preferred method from the zoo of optimization algorithms, see, e.g., \cite{AbsilMahonySepulchre:2008}, and
retractions could be used to replace the Riemannian exponential.
For THI, the values of the coefficient functions at $\omega^*$ are obtained via solving $k(d+1)$ Euclidean interpolation problems in a selected tangent space.
The back-mapping to the manifold must be via the Riemannian exponential, for a retraction would compromise the interpolation condition, unless the inverse of the same retraction is also used at the preprocessing stage.

Table \ref{tab:CompMethods} summarizes the main features of both approaches.
\begin{table}[tbp]
    \begin{tabular}{ccc}
    \toprule
                            & Barycentric           & Tangent space\\ \midrule
\#(coefficient functions)   & $k$           & $k(d+1)$\\
minimum \#(samples)         & $k>\dim(\mcM) +1$ & no restriction \\
base-point dependent        & no            & yes \\
manifold data processing via&  $\Log_q$     &$\Log_q$,  $d(\Log_q)_p$\\
model evaluation via        &Riemann. optimization & interp. in tangent space\\
\bottomrule
    \end{tabular}
    \caption{Comparison of the barycentric and the tangent space approach to Hermite manifold interpolation for a $d$-variate interpolation problem with a number of $k$ sample locations $\omega_j\in \R^d$.}
    \label{tab:CompMethods}
\end{table}
%
\section{Numerical examples}
\label{sec:numex}
In this section, we demonstrate multivariate Hermite manifold interpolation by means of academic examples.
For clarity of the exposition, but for no mathematical reasons, we focus in the experiments on the case, where $\mcM$ is an embedded submanifold of $\R^{N}$ and we express data points $p\in\mcM$ and tangent vectors $v\in T_p\mcM$ in \emph{extrinsic coordinates}. This means that we use coordinates of $\R^N$ rather than local coordinates to address locations on $\mcM$. This situation is also to be expected in practical applications.
\begin{example*}
On the unit sphere $\mcM=S^2=\{p\in\R^3\mid p_1^2+p_2^2 +p_3^2 =1\}\subset \R^3$, which is of (manifold) dimension $m=2$, 
  points are addressed by their three Euclidean coordinates $p=(p_1,p_2,p_3)^{\mathrm{T}}$.
Likewise, tangent vectors $v\in T_pS^2=\langle p\rangle^\bot$ are also given by extrinsic coordinates $v=(v_1,v_2,v_3)^{\mathrm{T}}\in T_pS^2\subset \R^3$.
\end{example*} 
Consider a differentiable, $d$-variate, $\mcM$-valued function
\begin{equation*}
    f\colon\R^d\supset D\to \mcM\subset \R^N, \omega \mapsto f(\omega)
\end{equation*}
Suppose that $k$ sample locations $\omega^j=(\omega^j_1,\ldots,\omega^j_d)^{\mathrm{T}}\in \R^d$ are selected
and that sample points $p^j =(p^j_1,\ldots,p^j_N)^{\mathrm{T}} = f(\omega^j)\in\mcM\subset \R^N$ and tangent vectors $v^i_j=\partial_i f(\omega^j)\in T_{p_j}\mcM\subset\R^N,i=1,\ldots,d$ are available.\\
For THI, the coefficient weight functions are obtained by Hermite interpolation of the data sets listed in \eqref{eq:THI_sample_data}.
In the BHI approach, we construct interpolants by computing numerical solutions to the Riemannian optimization problem \eqref{eq:baryInterp}.
In this case, the vector of sample values for the $j$th weight function $\varphi_j$ in \eqref{eq:baryInterp} is
\begin{equation}
\label{eq:sample_weight_funs}
    (\varphi_j(\omega^1), \ldots, \varphi_j(\omega^k)) = e_j = (0,\ldots,0,\stackrel{j}{1},0,\ldots,0)\in  \R^k
\end{equation}
and is independent of the sample points $p_j\in \mcM$.
The vector of the partial derivatives of $\varphi_j$ at the sample locations is obtained from Algorithm \ref{alg:partial_derivs}.
\begin{algorithm}{Practical computation of the derivatives of the weight functions for BHI.}
\begin{algorithmic}[1]
\REQUIRE{Hermite data set consisting of sample pairs $ (p_j=f(\omega_j), \omega_j)\in \mcM \times\R^d$, $ j=1,\ldots,k$, and partial derivatives $v_l^i = \partial_i f(\omega_l)\in T_{p_l}\mcM$, $i=1,\ldots,d$, $l=1,\ldots,k$.
         }
    \FOR{$i=1,\ldots,d$} 
      \FOR{$l=1,\ldots,k$} 
        \STATE{Compute $\Log_{p_l}(p_j)\in T_{p_l}\mcM\subset \R^N$, $j=1,\ldots,k$, $j\neq l$}
       \STATE{$X_l\coloneqq \begin{pmatrix}[c|c|c|c|c|c]
         \Log_{p_l}(p_1) & \ldots & \Log_{p_l}(p_{l-1})& \Log_{p_l}(p_{l+1}) &\ldots &\Log_{p_l}(p_k)
       \end{pmatrix}\in \R^{N\times (k-1)}$}
       \STATE{Solve $
       \begin{pmatrix}
         X_l\\
         1,\ldots,1
       \end{pmatrix} 
      \begin{pmatrix}
         c_1\\
         \vdots\\
         c_{k-1}
       \end{pmatrix}
       =
       \begin{pmatrix}
         v^i_l\\
         0
       \end{pmatrix}$}
       \STATE{ $\left(
         \partial_i\varphi_1(\omega_l),
         \ldots,
         \partial_i\varphi_{l-1}(\omega_l),
         \partial_i\varphi_l(\omega_l),
         \partial_i\varphi_{l+1}(\omega_l),
         \ldots,
         \partial_i\varphi_k(\omega_l)
       \right) \coloneqq 
       \left(
         c_1,
         \ldots,
         c_{l-1},
         0,
         c_l,
         \ldots,
         c_{k-1}
       \right)$}
      \ENDFOR
    \ENDFOR
\ENSURE{Partial derivatives of the weight interpolation functions at the sample locations $\partial_i\varphi_{j}(\omega_l)$, $j,l=1,\ldots,k$, $i=1,\ldots,d$.}
\end{algorithmic}
\label{alg:partial_derivs}
\end{algorithm}
\begin{remark}
In the upcoming experiments, we face the setting, that $\text{dim}(\mcM)=m<k-1$, where $k$ is the number of sample points. Hence, the linear system in step 5 of Alg. \ref{alg:partial_derivs} is underdetermined. The bottom row of all ones is added to the system matrix to enforce that $\sum_j \partial_i\varphi_{j}(\omega_l) = 0$.\\
The minimum 2-norm solution to the system is obtained via the pseudo-inverse based on the SVD, see
\cite[\S 5.7, p. 270--273]{GolubVanLoan:1996}.
Let $r = \rank X$
and let $U_r\Sigma_r V^{\mathrm{T}}_r = X$
be the reduced SVD, i.e., 
$U_r\in \R^{N\times r}$,
$\Sigma_r\in\R^{r\times r}$,
$V_r\in \R^{(k-1)\times r}$. 
The columns of $U_r$ form a basis for the column space of $X$.
The underdetermined equation $Xc = v_l^i$ corresponds to an equation for the coordinates with respect to the basis given by $U_r$, namely,
$(U_r^{\mathrm{T}} X) c= (\Sigma_r V_r^{\mathrm{T}}) c = \tilde{v}_l^i = U_r^{\mathrm{T}}v_l^i$.
Writing $\bone\in \R^{k-1}$ for the vector with all entries equal to $1$, we enforce that the derivatives of the coefficient functions sum up to zero via
\[
  \tilde{X}_r c :=
    \begin{pmatrix}
      \Sigma_r V_r^{\mathrm{T}} \\
      \bone^{\mathrm{T}}
    \end{pmatrix} c
    =    \begin{pmatrix}
       \tilde{v}_l^i\\
      0
    \end{pmatrix}.
\]
The minimum 2-norm solution is obtained via $c = \tilde{X}_r^{\mathrm{T}}(\tilde{X}_r\tilde{X}_r^{\mathrm{T}})^{-1}\tilde{v}_l^i$. Using the SVD data of $X$ and Schur complement inversion, this can be calculated in closed form
\begin{equation}
\label{eq:underdetX_closed_form}
    c = x + \frac{\langle\bone, x\rangle}{s} \left(V_r (V_r^{\mathrm{T}}\bone) - \bone \right),
    \quad
    x = V_r\Sigma_r^{-1} \tilde{v}_l^i \in \R^{k-1},
    \quad
    s = (k-1) - \lVert V_r^{\mathrm{T}}\bone\rVert_2^2\in\R.
\end{equation}
\end{remark}
For constructing the interpolation weight functions $\varphi_j$ in BHI and $\varphi_j, \psi_{i,l}$ in THI, we use the method of gradient-enhanced Kriging with the cubic correlation model and a
fixed correlation hyperparameter vector $\theta = (\theta_1,\theta_2) = (0.5, 0.5)$.
The cubic correlation model is
$\rho(\theta,\omega^i,\omega^j)= \prod_{l=1}^{2}{\rho_l(\theta_l, (\omega^i_l - \omega^j_l))}$,
where 
\[
\rho_l(\theta_l, (\omega^i_l - \omega^j_l)) =
\left\{
\begin{array}{ll}
	1-3\left(\theta_l |\omega^i_l - \omega^j_l|\right)^2 + 2 \left(\theta_l |\omega^i_l - \omega^j_l|\right)^3, &  \theta_l |\omega^i_l - \omega^j_l|<1\\
	0, & \theta_l |\omega^i_l - \omega^j_l|\geq 1
\end{array}
\right. .
\]
For the details, we refer to \cite[Section 7.2]{ForresterSobesterKeane:2008} and \cite{Zimmermann:2013}.
The input sample values are given by \eqref{eq:sample_weight_funs} and Algorithm \ref{alg:partial_derivs} for BHI and by \eqref{eq:THI_sample_data} for THI,  respectively.
BHI requires to solve the optimization problem \eqref{eq:baryInterp} and thus to provide the optimizer with an initial guess.
In the first run of the optimizer, we use the first sample location as the initial guess, i.e., $q_0 = p_1$. For every consecutive run, we use the optimized solution $q^*$ from the previous run as the starting point for the next optimization procedure.
%
%
\subsection{Interpolation of the Gauß map of the Helicoid}
\label{sec:numex_Helicoid}
In this section we consider an academic example on the unit sphere $\mcM=S^2$.
The Riemannian exponential and logarithmic maps on $S^2$ are
\begin{align*}
    \Exp_q(v) &= \cos(\lVert v\rVert_2)q +\sin(\lVert v\rVert_2) \frac{v}{\lVert v\rVert_2}\in S^2,
    \\
    \quad \Log_q(p) &= \arccos(\langle q,p\rangle)
    \frac{p-\langle q,p\rangle q}{\lVert p-\langle q,p\rangle q\rVert_2}\in T_qS^2,
\end{align*}
respectively.
As a test function, we use the Gauß map of the helicoid in $\R^3$,
\begin{equation*}
    f\colon [a,b]^2\to S^2, \quad (\omega_1,\omega_2) \mapsto \frac{1}{(e^{2\omega_1} +1)}
    \begin{pmatrix}
       2e^{\omega_1}\cos(\omega_2)\\
       2e^{\omega_1}\sin(\omega_2)\\
       e^{2\omega_1} - 1
    \end{pmatrix}.
\end{equation*}
This function yields the normal field of the helicoid and is obtained from the stereographic projection of $(\omega_1,\omega_2)\mapsto \exp(\omega_1 +i\omega_2)\in \C$ onto $S^2$.\\
The partial derivatives of $f$ are
\begin{align*}
\partial_1 f(\omega_1,\omega_2) &= \frac{-2e^{2\omega_1}}{(e^{2\omega_1}+1)^2} 
        \begin{pmatrix}
          2e^{\omega_1}\cos(\omega_2) \\ 2e^{\omega_1}\sin(\omega_2) \\  e^{2\omega_1} - 1
        \end{pmatrix}
        + \frac{2}{e^{2\omega_1}+1} \begin{pmatrix}
          e^{\omega_1}\cos(\omega_2) \\ e^{\omega_1}\sin(\omega_2) \\ e^{2\omega_1}
        \end{pmatrix}, \\
  \partial_2 f(\omega_1,\omega_2) &= \frac{1}{e^{2\omega_1}+1} \begin{pmatrix}
    -2e^{\omega_1}\sin(\omega_2) \\ 2e^{\omega_1}\cos(\omega_2)\\ 0
    \end{pmatrix}.
\end{align*}

%
We sample the function values and partial derivatives of $f$ on $k=9$ uniformly distributed sample locations $(\omega^j_1,\omega^j_2)\in [a,b]^2 = \left[-\frac14\pi, \frac14\pi\right]^2$, $j=1,\ldots,9$ and conduct BHI and THI.

The numerical gradient descent method of Algorithm \ref{alg:BaryInt} is conducted with a fixed step size of $\alpha =1.0$ and a numerical convergence threshold of $\tau = 1.0\cdot 10^{-8}$. For THI, the Riemannian barycenter of the sample data set is used as the tangent space base point.
The respective interpolant is evaluated on a uniform grid of $101\times 101$ points in the same domain $[a,b]^2$.
%

\begin{table}[tbp]
\begin{small}
    \centering
    \begin{tabular}{cccccc}\toprule
    \multicolumn{6}{c}{\emph{Parameter settings}}\\
     Manifold               & \#{variables}              & domain                   & \#{samples} & threshold                & line step\\
     $S^2$                  & $d=2$                     & $[-\frac14\pi, \frac14\pi]^2$ &  $k=9$           & $\tau =1.0\cdot 10^{-8}$       & $\alpha = 1.0$\\[\baselineskip]
    \midrule
    & \\
     \midrule
     \multicolumn{6}{c}{\emph{Results: barycentric Hermite interpolation (BHI)}}\\
      \multicolumn{2}{c}{Wall clock time} & \multicolumn{2}{c}{Interpolation error} & \multicolumn{2}{c}{FD error}\\
      \cmidrule(r){1-2}\cmidrule(lr){3-4}\cmidrule(l){5-6}
     offline & online & max & avg & avg.~$\partial_1 \hat f(\omega^j) $ & avg.~$\partial_2 \hat f(\omega^j)$\\
     $7.2\cdot 10^{-3}$s & $6.42\cdot 10^{-5}$s                    &  $2.14\cdot 10^{-2}$          &$8.23\cdot 10^{-3}$& $6.17\cdot 10^{-6}$           & $6.27\cdot 10^{-6}$\\
       \midrule
      \multicolumn{6}{c}{\emph{Results: tangent space Hermite interpolation (THI)}}\\
      \multicolumn{2}{c}{Wall clock time} & \multicolumn{2}{c}{Interpolation error} & \multicolumn{2}{c}{FD error}\\
      \cmidrule(r){1-2}\cmidrule(lr){3-4}\cmidrule(l){5-6}
     offline & online & max & avg & avg.~$\partial_1 \hat f(\omega^j) $ & avg.~$\partial_2 \hat f(\omega^j)$\\
     $2.5\cdot 10^{-2}$s & $1.5\cdot 10^{-4}$s                    &  $2.36\cdot 10^{-3}$          &$1.10\cdot 10^{-3}$& $6.47\cdot 10^{-6}$           & $7.61\cdot 10^{-6}$\\
    \bottomrule
    \end{tabular}
    \caption{Parameter settings and interpolation results for the experiment of Subsection \ref{sec:numex_Helicoid} associated with Figure \ref{fig:Error_gauss_3x3}.
    The `offline' time refers to the construction of the interpolant from the sampled data, while the `online' time accounts for querying the interpolant at a trial location.}
    \label{tab:numex1}
\end{small}
\end{table}
Table \ref{tab:numex1} summarizes the parameter settings for this experiment and lists the wall clock time and the interpolation errors.
The wall clock time is split into an offline stage and an online stage.
The offline stage consists of the construction of the Hermite interpolants for the weight coefficient functions $\varphi_j$ for BHI and $\varphi_j, \psi_{i,l}$ for THI, respectively. The online stage accounts for querying the interpolant at a given parameter location.
The evaluation times are averaged over the number of $10,201= 101^2$ runs.\\
The sample data set in form of surface normal vectors, their partial derivatives and the interpolation error surfaces are displayed in Figure \ref{fig:Error_gauss_3x3}.
\begin{figure}[tbp]
    \centering
    \includegraphics[width=1.0\textwidth]{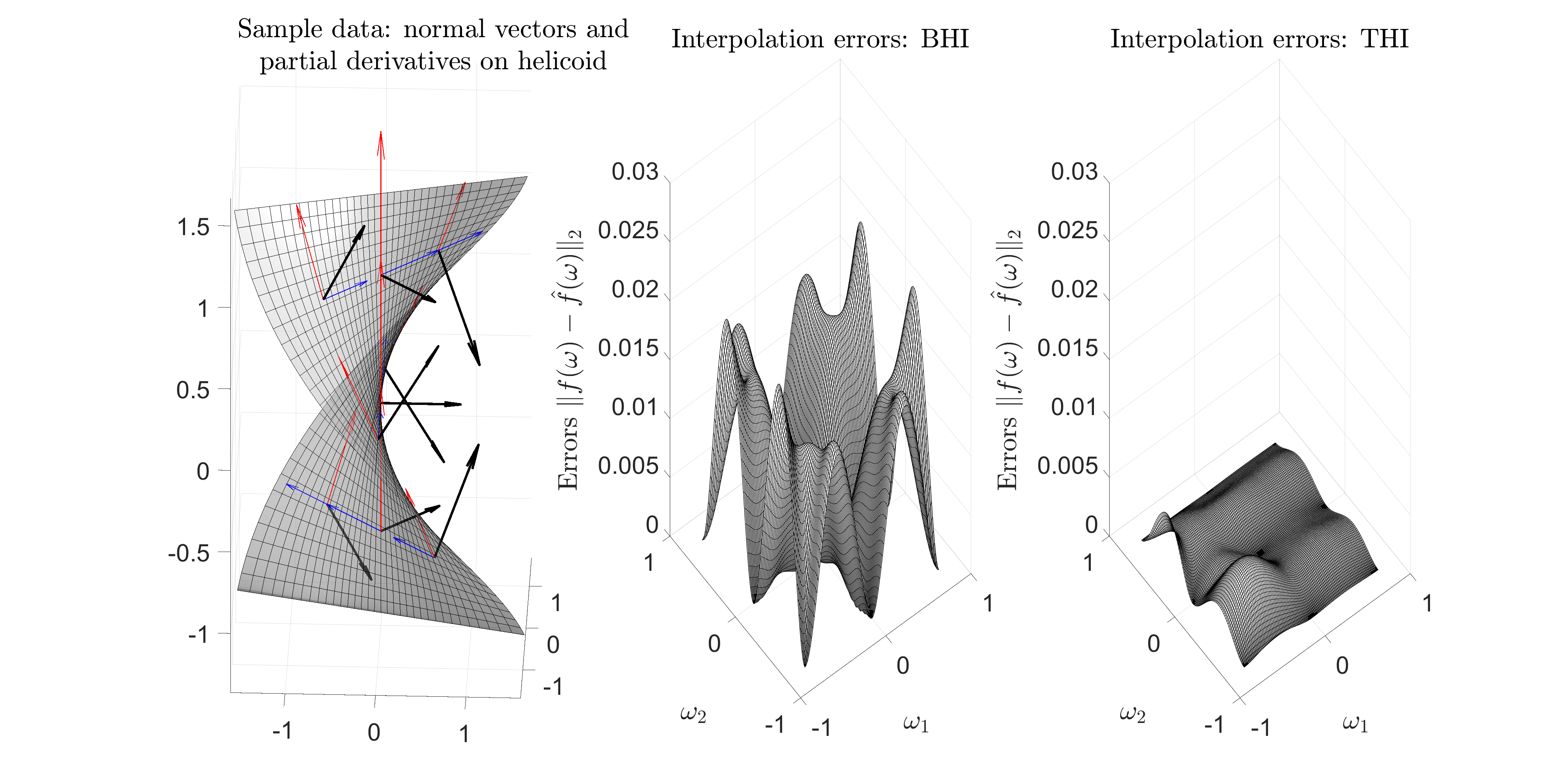}
    \caption{Interpolating the Gauß map of the Helicoid from Subsection \ref{sec:numex_Helicoid}. Left: sample data set and partial derivatives.
    Middle: error surface obtained with BHI. Right: error surface obtained with THI.
    }
    \label{fig:Error_gauss_3x3}
\end{figure}

As can be seen from the Table \ref{tab:numex1} and Figure \ref{fig:Error_gauss_3x3}, the averaged accuracy of the THI interpolant is roughly a factor of $6$ times better than the averaged accuracy of the BHI interpolant. A finite difference check shows that both methods provide interpolants that meet the sampled tangent vectors.
The computation time for the offline stage under the THI approach is $3.5$ times higher than that for BHI.
This is expected, since for THI, the data preprocessing involves evaluating the Riemann exp and log maps plus the differential of the latter for moving the sampled derivatives to the same tangent space.\footnote{For the academic case at hand, a closed form calculation of $\diff{}(\Log_q)_p$ is possible. However, for consistency with the general case, in the numerical experiments, we work with the finite difference approximation of \eqref{eq:THI_FD_prep}.} Moreover, a number of 
$k(d+1)=27$ Euclidean Hermite interpolation problems have to be solved to obtain the interpolated weight functions. In contrast, for BHI, only $k=9$ coefficient functions have to be fitted. This, however, involves solving an underdetermined linear equation system.\\
In the online stage of BHI, the basic gradient descent converges after an average number of $4.9$ iterations with the worst case taking $7$ iterations.
On average, for the test case at hand, BHI is online $2.3$ times faster than THI. However, too much weight should not be given to the timing results, as they are highly dependent on the chosen optimization algorithm and choosing a `best one' is beyond the scope of these experiments.\\
The THI method produces different interpolants for different choices of the tangent space base point.
To illustrate this issue, Figure \ref{fig:Error_base_popint_3x3} shows the error surfaces corresponding to three different base-point selections.
The left picture Figure \ref{fig:Error_base_popint_3x3} is produced by working in $T_{p_1}S^2$, 
the middle picture by working in $T_{p_4}S^2$,
and the right picture corresponds to working in $T_{p_9}S^2$, where $p_j = f(\omega_j)$.
\begin{figure}[tbp]
    \centering
    \includegraphics[width=1.0\textwidth]{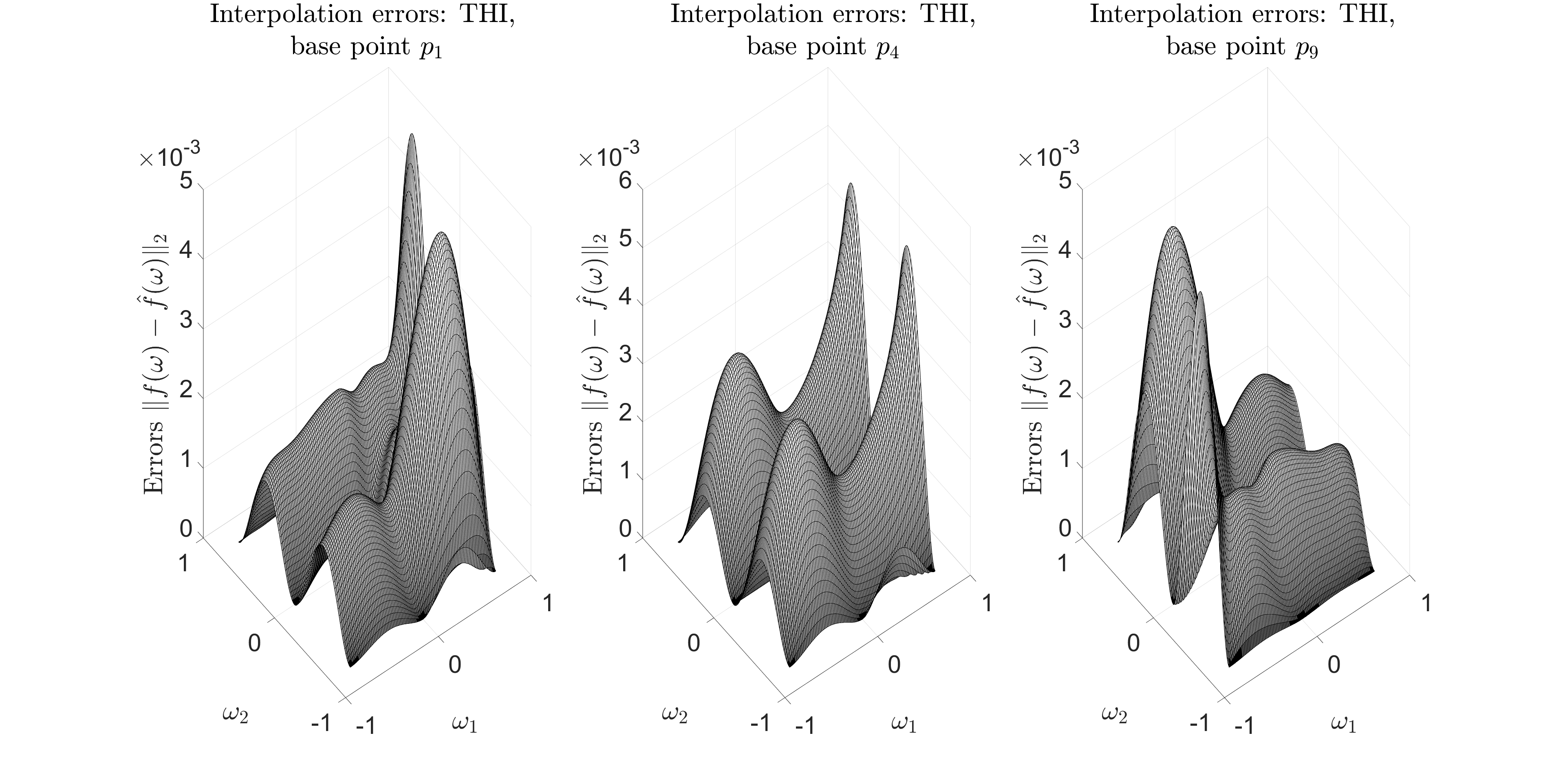}
    \caption{Error surfaces under the THI method from Subsection \ref{sec:numex_Helicoid} for different base tangent spaces.
    }
    \label{fig:Error_base_popint_3x3}
\end{figure}
%
\subsection{Interpolation on \texorpdfstring{$SO(3)$}{SO(3)}}
\label{sec:numex_SO3}
The next test case considers a Hermite data set on the special orthogonal group $SO(3) = \{Q\in \R^{3\times 3}\mid Q^{\mathrm{T}}Q = I, \det(Q) = 1\}$.
The tangent space of $SO(3)$ at a point $Q\in SO(3)$ is 
$T_QSO(3) = \{V\in\R^{3\times 3}\mid Q^{\mathrm{T}}V +V^{\mathrm{T}}Q = 0\}$. The Riemannian $\exp$ and $\log$ maps on $SO(3)$ are
\begin{align*}
\Exp_Q\colon& T_QSO(3) \to SO(3),& V\mapsto Q\exp_m(Q^{\mathrm{T}}V),\\
\Log_Q\colon&  \mathcal{D}_Q \to T_QSO(3),& R\mapsto Q\log_m(Q^{\mathrm{T}}R),
\end{align*}
where $\exp_m$ and $\log_m$ denote the classical matrix exponential and logarithm functions, see \cite{Higham:2008}, and
$\mathcal{D}_Q\subset SO(3)$ is a domain around $Q$ such that for all $R\in \mathcal{D}_Q$, the orthogonal matrix $Q^{\mathrm{T}}R$ does not feature $\lambda = -1$ as an eigenvalue.
As a test function, we consider
\begin{align*}
    f\colon &[a,b]^2\to SO(3), \quad (\omega_1,\omega_2) \mapsto 
    \exp_m X(\omega_1, \omega_2), \text{ where}\\
    &\quad X(\omega_1, \omega_2) = 
    \begin{pmatrix}[ccc]
       0                                  & \omega_1^2 + \frac12\omega_2 & \sin\bigl(4\pi(\omega_1^2+\omega_2^2)\bigr)
       \\[.66em]
      -\omega_1^2 - \frac12\omega_2       &      0                       & \omega_1 + \omega_2^2
      \\[.66em]
      -\sin(4\pi(\omega_1^2+\omega_2^2)) & -\omega_1 - \omega_2^2       & 0
    \end{pmatrix}.
\end{align*}
The sample locations $P_j=\exp_m X(\omega^j)$ at $\omega^j=(\omega_1^j, \omega_2^j)$ and the corresponding partial derivatives $V^i_j =\deriv[t]\big\vert_{t=0} \exp_m(X(\omega^j+t e_i))= \diff{}(\exp_m)(X(\omega^j))[\partial_i X(\omega^j)] $, $ i=1,2$ of the test function can be obtained 
by Mathias' theorem, see \cite[Thm. 3.6]{Higham:2008}:
\begin{equation*}
    \exp_m\begin{pmatrix}[cc]
      X(\omega^j) & \partial_i X(\omega^j)
     \\[.66em]
      0         &  X(\omega^j)
    \end{pmatrix}
    =
    \begin{pmatrix}[cc]
      \exp_m(X(\omega^j)) & \diff{}(\exp_m)(X(\omega^j))[\partial_i X(\omega^j)]
     \\[.66em]
      0         &  \exp_m(X(\omega^j))
    \end{pmatrix}.
\end{equation*}
For BHI, the partial derivatives of the interpolation weight functions $\varphi_j$ are computed with Algorithm \ref{alg:partial_derivs},
where we use the vectorized tangent matrices $\text{vec}(\Log_{P_l}(P_j))\in\R^{3\cdot 3}$ as columns to form the matrices $X_l$. The input $V_l^i$ on the right hand side of the equation system in step 5 of Algorithm \ref{alg:partial_derivs} is vectorized accordingly. For BHI, for each trial location $\omega^*$, the interpolant $\hat f(\omega^*)$ is computed with Algorithm \ref{alg:BaryInt}, while for THI, it is computed according to \eqref{eq:tanginterp}
and \eqref{eq:THI_interpolant}. FOR THI, the Riemannian barycenter of the sample data set is used as the tangent space base point.\\
For quantifying the accuracy of the interpolation, we compute the relative errors 
$\frac{\lVert f(\omega) - \hat f(\omega)\rVert_F}{\sqrt{3}}$. Mind that $\lVert Q\rVert_F= \sqrt{3}$ for any matrix $Q\in SO(3)$.
\begin{table}[tbp]
	\begin{small}
    \centering
    \begin{tabular}{cccccc}\toprule
    \multicolumn{6}{c}{\emph{Parameter settings}}\\
     Manifold               & \#{variables}              & domain                  & \#{samples} & threshold                & line step\\
     $SO(3)$                  & $d=2$                     & $[0.5, 0.5]^2$ &  $k=49$ (Cheby.)          & $\tau =1.0\cdot 10^{-6}$       & $\alpha = 1.0$\\[\baselineskip]
    \midrule
    & \\
     \midrule
    \multicolumn{6}{c}{\emph{Results: barycentric Hermite interpolation (BHI)}}\\
      \multicolumn{2}{c}{Wall clock time} & \multicolumn{2}{c}{Interpolation error} & \multicolumn{2}{c}{FD error}\\
      \cmidrule(r){1-2}\cmidrule(lr){3-4}\cmidrule(l){5-6}
     offline & online & max & avg & avg.~$\partial_1 \hat f(\omega^j) $ & avg.~$\partial_2 \hat f(\omega^j)$\\
     $0.41$s & $0.077$s                    &  $0.029$          &$0.0069$ & $4.45\cdot 10^{-4}$           & $4.91\cdot 10^{-4}$\\
       \midrule
    \multicolumn{6}{c}{\emph{Results: tangent space Hermite interpolation (THI)}}\\
      \multicolumn{2}{c}{Wall clock time} & \multicolumn{2}{c}{Interpolation error} & \multicolumn{2}{c}{FD error}\\
      \cmidrule(r){1-2}\cmidrule(lr){3-4}\cmidrule(l){5-6}
     offline & online & max & avg & avg.~$\partial_1 \hat f(\omega^j) $ & avg.~$\partial_2 \hat f(\omega^j)$\\
     $0.73$s & $0.0023$s                    &  $0.027$          &$0.0065$ & $3.9\cdot 10^{-4}$           & $4.6\cdot 10^{-4}$\\
    \bottomrule
    \end{tabular}
    \caption{Parameter settings and interpolation results for the experiment of Subsection \ref{sec:numex_SO3}
    associated with Figure \ref{fig:Error_SO3_both}.
    The `offline' time refers to the construction of the interpolant from the sampled data, while the `online' time accounts for querying the interpolant at a trial location.}
    \label{tab:numex_SO3}
    \end{small}
\end{table}
%
%
%
%
%
%
%
%
%
%
%
%
In this experiment, we rely on a two-dimensional Chebychev sample plan
\begin{equation*}
    \left\{\frac12 (b-a)\cos\left(\frac{(2j-1)\pi}{2k}\right) + \frac12(b+a)\mid j=1,\ldots, k\right\}^2\subset \R^2.
\end{equation*}
\begin{figure}[tbp]
    \centering
    \includegraphics[width=.4\textwidth]{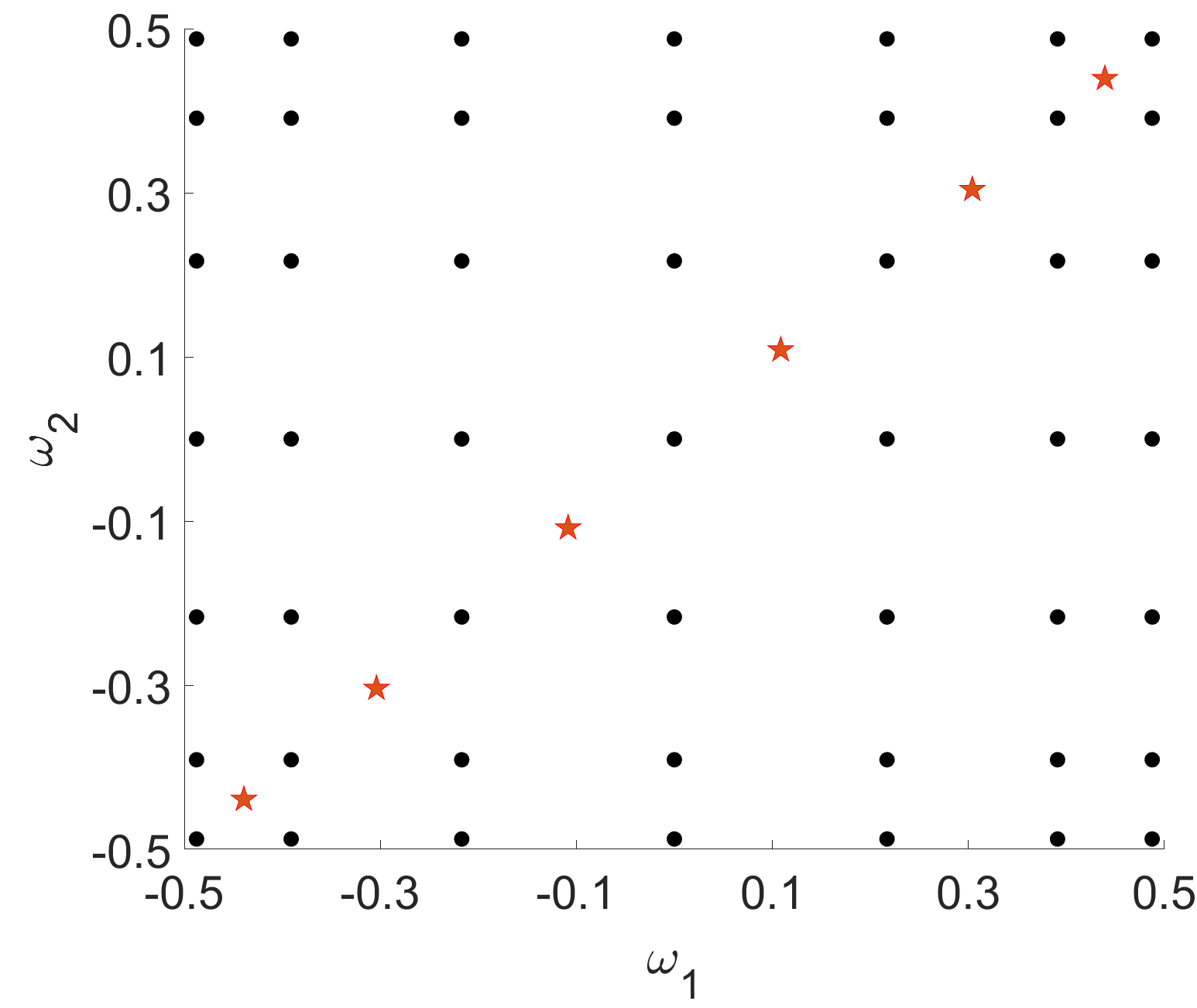}
    \caption{Sample plan associated with the experiment of Section \ref{sec:numex_SO3}. 
    Black dots: Chebychev $7\times7$ grid on the domain $[-0.5, 0.5]^2$. The stars indicate the trial locations that are used for visualization purposes in the upcoming Figure \ref{fig:t-pots}.}
    \label{fig:chebysampling}
\end{figure}
We sample the function values and partial derivatives of $f$ on $k=49$ Chebychev sample locations $(\omega^j_1,\omega^j_2)\in [a,b]^2 = \left[-0.5, 0.5\right]^2$, $j=1,\ldots,49$, see  \Cref{fig:chebysampling}.
The numerical gradient descent method of Algorithm \ref{alg:BaryInt} is conducted with a fixed step size of $\alpha =1.0$ and a numerical convergence threshold of $\tau = 1.0\cdot 10^{-6}$.
The BHI and THI interpolants are evaluated on a uniform grid of $76\times 76$ trial points in the same domain.
The corresponding interpolation error surfaces are displayed in 
\Cref{fig:Error_SO3_both}.
\Cref{tab:numex_SO3} summarizes the parameter settings as well as the timing and accuracy results.\\
In the test case at hand, BHI and THI produce interpolants of roughly the same averaged accuracy, with THI performing slightly better than BHI.
To the naked eye, the error surfaces look identical for both approaches.
A finite difference check shows again that both methods provide interpolants that meet the sampled tangent vectors.
The computation time for the offline stage under the THI approach is roughly $2$ times higher than that for BHI.
In the online stage of BHI, the basic gradient descent converges after an average number of $18.0$ iterations with the worst case taking $87$ iterations.
On average, for the test case at hand, THI is online $33$ times faster than BHI. 
We emphasize again that the timing results are expected to vary considerably depending on which optimization algorithm is employed.

\begin{figure}[tpb]
    \centering
    \includegraphics[width=1.0\textwidth]{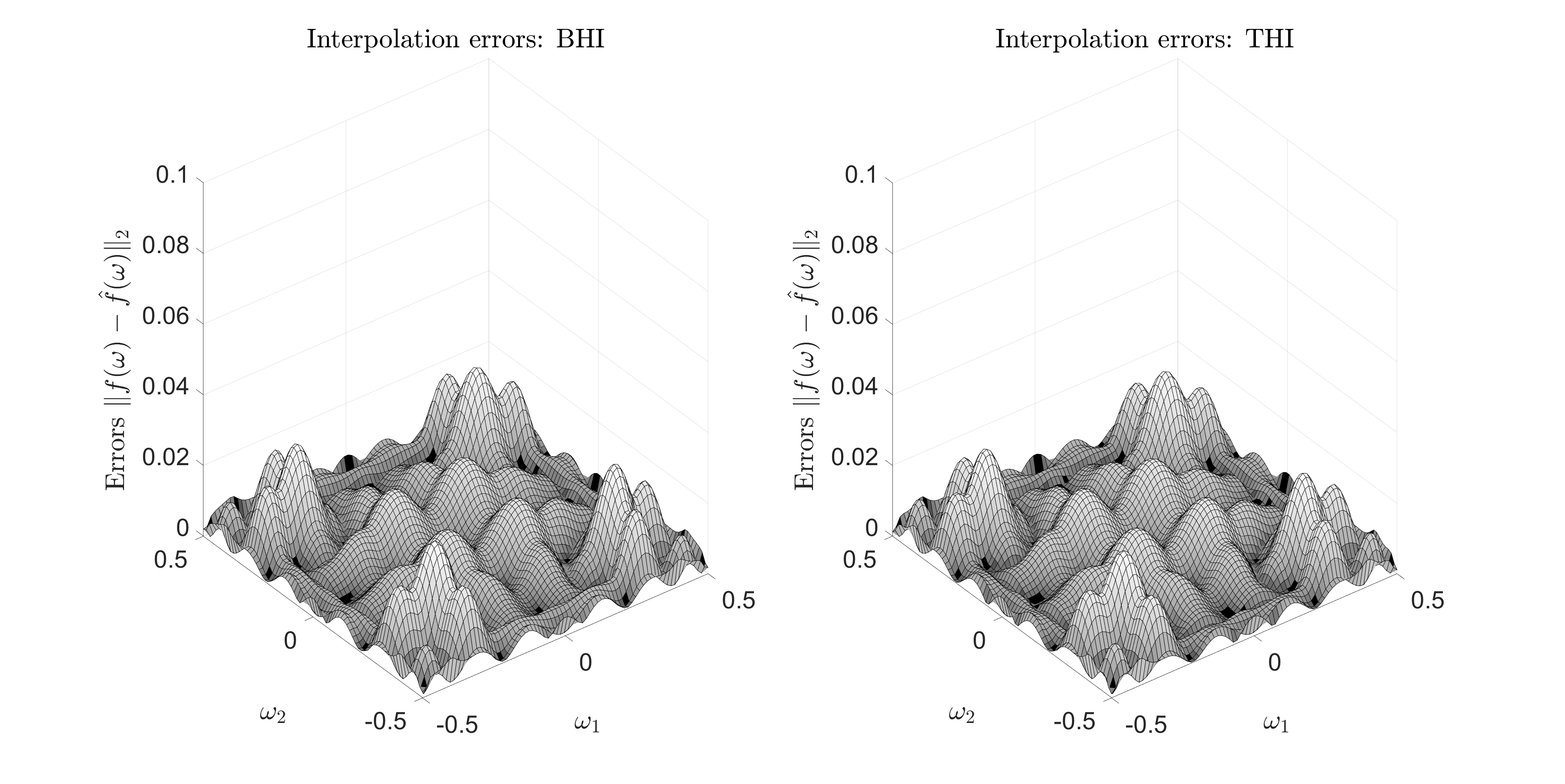}
    \caption{ 
    Error surfaces for $SO(3)$-interpolation on a Chebychev $7\times 7$ grid from Subsection \ref{sec:numex_SO3}. Left: Barycentric Hermite Interpolation (BHI).  Right: Tangent Space Hermite Interpolation (THI).}
    \label{fig:Error_SO3_both}
\end{figure}
In this test case, each interpolated value $\hat P  = \hat f(\omega)$ is a matrix in $SO(3)$. In order to further visualize the interpolation results,
\Cref{fig:Coords_SO3_cheby} displays the interpolated matrix component functions
$\omega\mapsto \hat P_{11}$, $\omega\mapsto \hat P_{22}$, $\omega\mapsto \hat P_{33}$, and $\omega\mapsto \hat P_{31}$.
Because the plots for THI and BHI virtually coincide, only the results for BHI are shown.
\begin{figure}[tbp]
    \centering
    \includegraphics[width=1.0\textwidth]{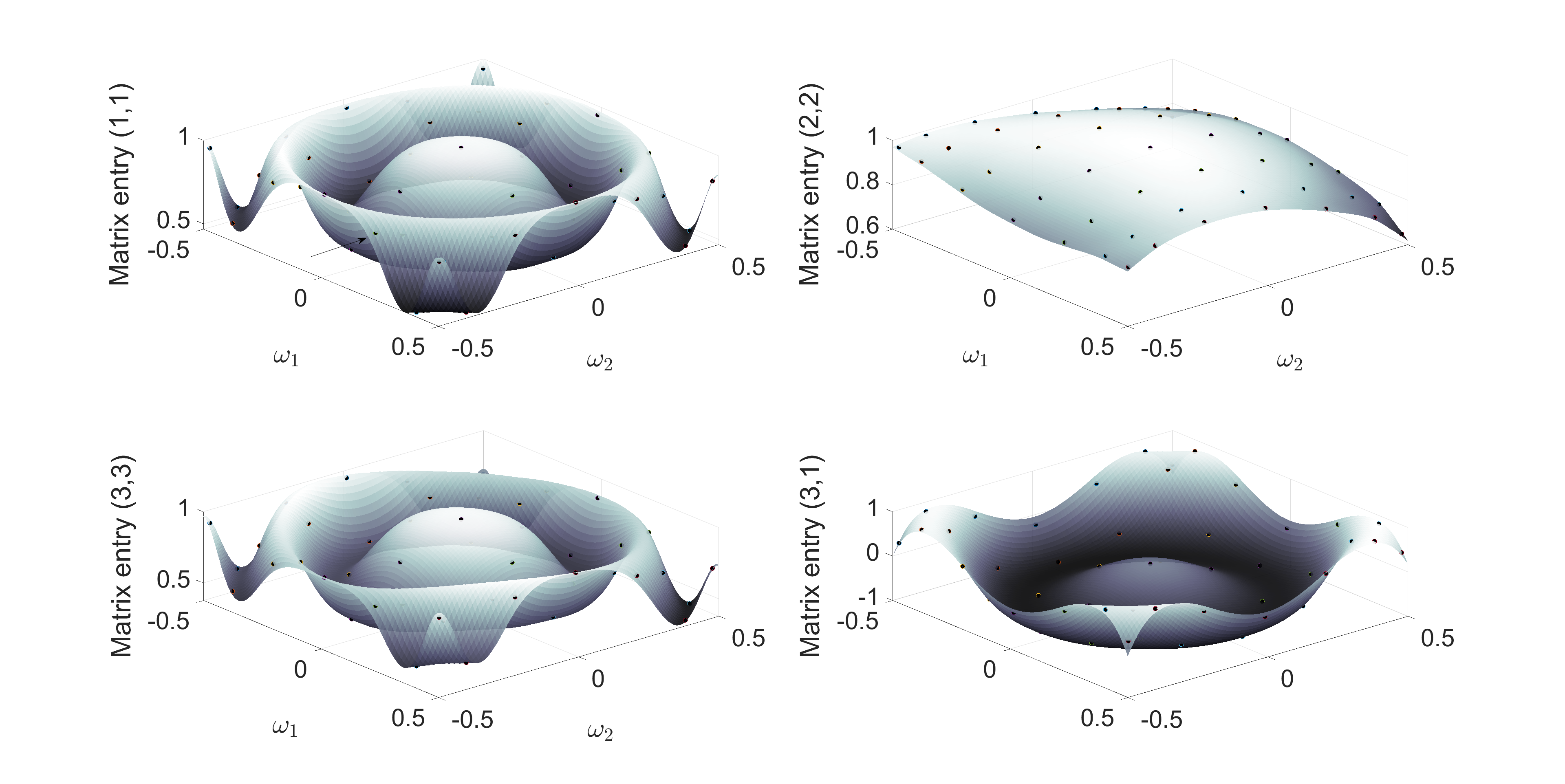}
    \caption{Plots of some selected interpolated matrix component functions 
    $(\omega_1, \omega_2) \to \left(\hat f(\omega_1, \omega_2)\right)_{i,j}\in \R$ of the experiment from Subsection \ref{sec:numex_SO3} . The black dots indicate the Chebychev $7\times 7$ sample grid.}
    \label{fig:Coords_SO3_cheby}
\end{figure}
The first diagonal component of the interpolant is juxtaposed with the corresponding component of the reference function in \Cref{fig:Coord11_SO3_cheby}.\\
\begin{figure}[tbp]
    \centering
    \includegraphics[width=0.7\textwidth]{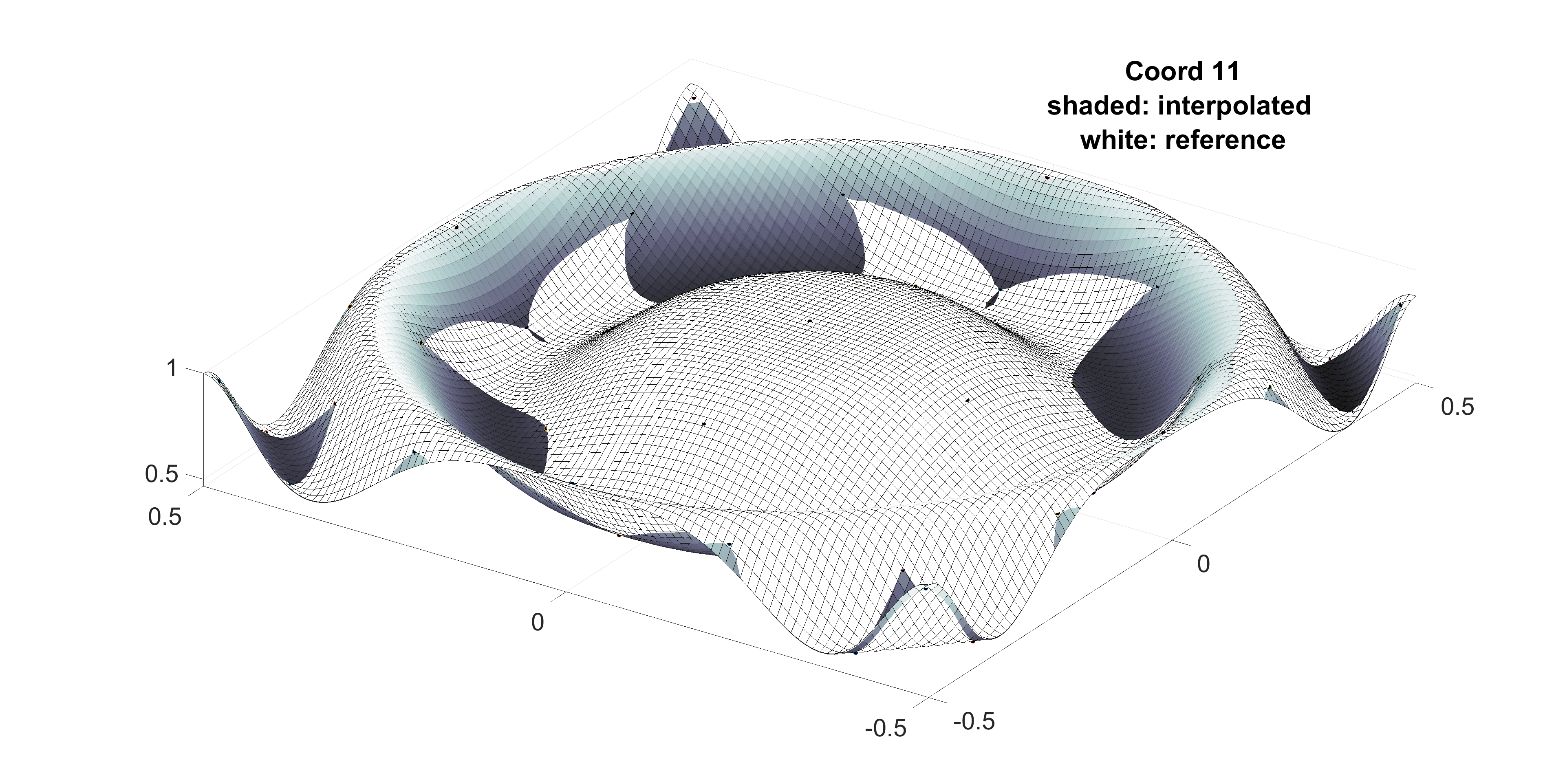}
    \caption{Interpolated matrix component function $\hat P_{11}=(\hat f(\omega))_{11}$ (shaded surface)
    and the reference matrix component  $P_{11}=f(\omega)$ (white surface) together with the sample locations on a Chebychev $7\times 7$ grid from Subsection \ref{sec:numex_SO3}.}
    \label{fig:Coord11_SO3_cheby}
\end{figure}
Matrices in $SO(3)$ induce rotations in the three-dimensional Euclidean space.
Such rotations may be visualized by their action on an object.
In \Cref{fig:t-pots}, we compare the Hermite  $SO(3)$-interpolant and the reference $SO(3)$-function via showing their actions on a tea pot object.
Both the interpolant and the reference function are queried at the parameter locations 
that are marked with a star in the sampling plan displayed in \Cref{fig:chebysampling}.
\begin{figure}[tpb]
    \centering
    \includegraphics[width=0.75\textwidth]{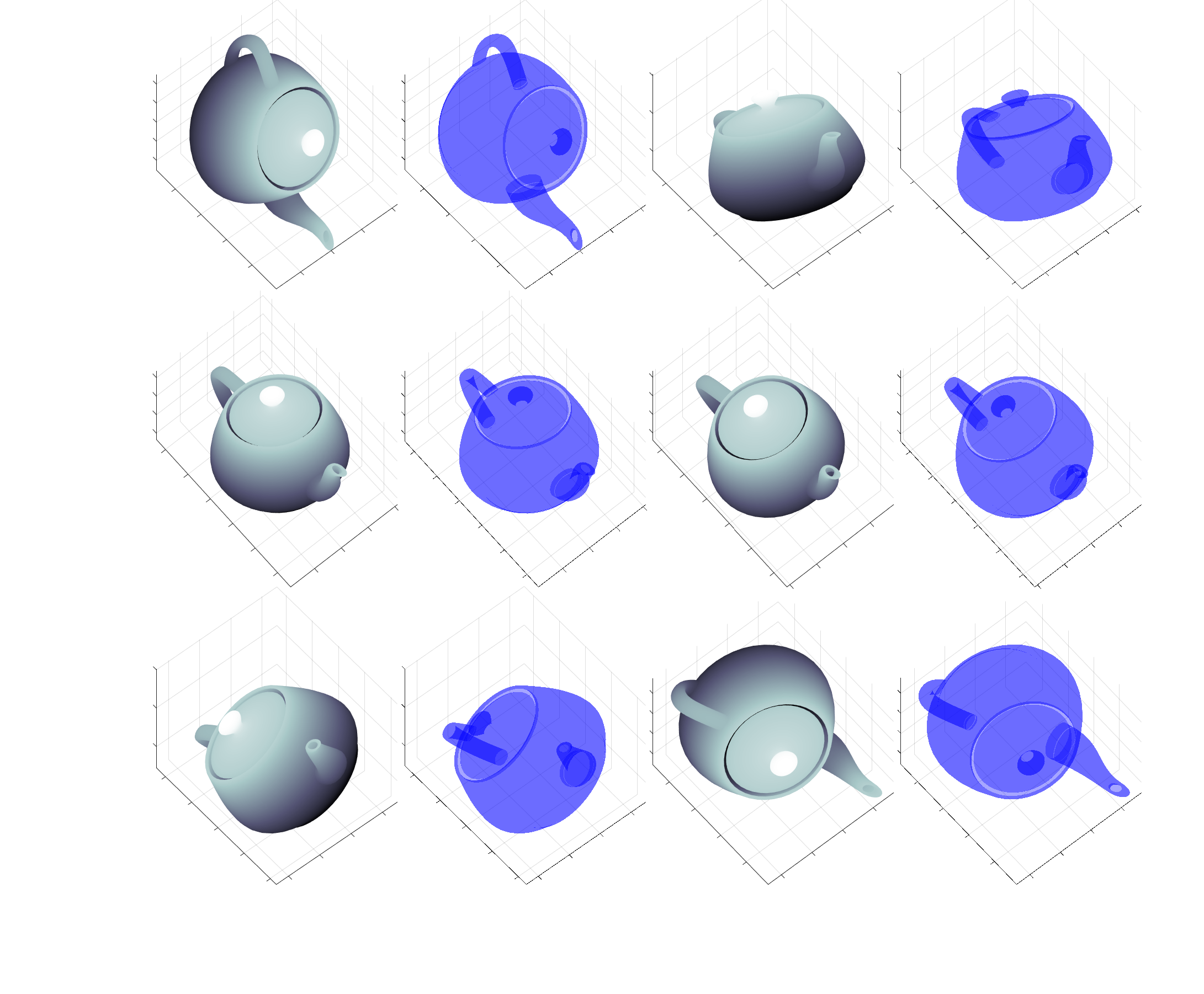}
    \caption{(from upper left to lower right): reference rotations (gray) and interpolated $SO(3)$-matrices (blue) at the 6 trial points displayed in Fig. \ref{fig:chebysampling}. The rotation matrices are visualized via their action on the tea pot object.}
    \label{fig:t-pots}
\end{figure}
The interpolated rotations visually agree very well with the reference rotations.
\section{Conclusions and discussion}
\label{sec:conclusions}
We have developed two approaches for multivariate Hermite interpolation of manifold-valued data sets.
On the one hand, Hermite interpolation can be achieved by computing correspondingly weighted Riemannian barycenters, referred to as \emph{barycentric Hermite interpolation} (BHI).
The method we presented should be seen as a generic framework for BHI.
For any practical application, the user has to
\begin{itemize}
    \item construct a sample plan, i.e.\ employ a method to select the sample points $\omega_j \in \R^d$ (design of experiment).
    \item choose a Riemannian optimization method to compute the weighted Riemannian center of mass in \eqref{eq:baryInterp}.
    (Gradient descend, Gauss-Newton-type methods, nonlinear conjugate gradients, etc.)
    \item choose interpolation weight functions $\omega \to \varphi_j(\omega)$ with the required properties.
    Note that the $\varphi_j$'s are scalar functions on a Euclidean domain so that this subtask does not require any Riemannian considerations.
    \item specify side constraints to obtain unique, smooth solutions to the underdetermined linear system in step 5 of Algorithm \ref{alg:partial_derivs}.
\end{itemize}
On the other hand, Hermite interpolation can be achieved by translating any classical approach of Hermite interpolation from Euclidean vector spaces to the tangent space of the manifold under consideration, referred to as \emph{tangent space Hermite interpolation} (THI). This requires to single out a manifold location and use it as the center for the tangent space. 
Moreover, all data, i.e., locations and tangent vectors have to be moved to the same tangent space. 
As a consequence, the interpolant depends on the base-point selection and the preprocessing stage is more involved.
Eventually, the interpolant is written as a weighted linear combination of the tangent space images of the sample points and the transported, sampled tangent space vectors. This entails that the number of weight functions in the THI approach equals the number of sample locations plus the total number of sampled partial derivatives. In a sense, THI can be considered a `brute-force' approach. Data is forced into a `Procrustean bed' (=the selected tangent space) and processed without regard to its nature.

From an aesthetic view point, the BHI method is the more appealing one: It is base-point independent and works with a number of weight coefficients that equals the number of sample points. Plus, it treats sample locations and tangent vectors, which in fact are incompatible entities, differently.
Yet, in the numerical experiments, the THI method proved to produce the more accurate results and the computational effort of querying  the interpolant 
in the online stage is much lower than for BHI, because the latter involves solving a Riemannian optimization problem at every trial point.
\section*{Acknowledgements}
The original idea for this research paper stems from a short-term visit of the first author to the "Numerical Algorithms and High-Per\-for\-mance Computing"-group at EPF Lausanne in Spring 2022.
The first author would like to thank the head of group Daniel Kressner and the doctoral assistant Axel E. J. S\'eguin
for their hospitality and for the very stimulating discussions during that visit.
%
%

%
\appendix
\section*{Appendix}
\section{Barycentric Hermite interpolation: The Euclidean case}
\label{app:EuclideanCase}
To ease the transition to the manifold case of barycentric interpolation, in this section, we outline multivariate, gradient-en\-han\-ced barycentric interpolation in Euclidean spaces.

Let $f\colon\R^d\to \R^m$ be differentiable and suppose that we have sampled data
$f(w_j) = p_j\in \R^m$, $j=1,\ldots,k$ and partial derivatives $\partial_i f(w_j) = v_j^i\in T_{p_j}\R^m\simeq \R^m$, $j=1,\ldots,k$, $i=1,\ldots,d$.
The task is to construct an interpolant $\hat f$ such that
 \begin{align}
 \label{eq:basicHermiteEuc}
 \hat f(w_j) &= p_j\in \R^m,  \quad j =1,\ldots,k\\
  \partial_i \hat f(w_j) &= v_j^i\in T_{p_j}\R^m,\simeq \R^m,  \quad j=1,\ldots,k; \hspace{0.2cm}  i=1,\ldots,d.
\end{align}
We approach this task by constructing weighted barycenters
\begin{equation}
\label{eq:baryEuc}
\hat f(\omega) = \text{arg}\min_{q\in\R^m} \frac12 \sum_{j=1}^k \varphi_j(\omega) \lVert q-p_j\rVert^2_2 \coloneqq \text{arg}\min_{q\in\R^m} L(q,\omega).
\end{equation}
Here, the $\varphi_j$, $j=1,\ldots,k$ are weight functions with the following properties
\begin{equation*}
 \varphi_j(\omega_i) = \delta_{ij}, \quad 
 \sum_{j=1}^k \varphi(\omega) \equiv 1.
\end{equation*}
The objective function $L(q,\omega)$ in \eqref{eq:baryEuc} is a weighted sum over squared distance terms
$\frac12 \dist(q,p_j)^2 = \frac12\lVert p_j -q\rVert_2^2 =: L_{p_j}(q)$.
The gradient is $\grad L_{p_j}(q) = (q - p_j)$. 
We emphasize that this is exactly minus the Riemannian logarithm $\Log_q(p_j) = p_j-q$ on the Euclidean $\R^m$. 
The minimizer of \eqref{eq:baryEuc}
is the unique zero of the gradient equation
\begin{equation}
\label{eq:grad0Euc}	
  0 = \grad_q L(q, \omega) = \sum_{j=1}^k \varphi_j(\omega)\grad L_{p_j}(q) = \sum_{j=1}^k \varphi_j(\omega) (q-p_j). 
\end{equation}
Because of $\sum_{j=1}^k \varphi_j(\omega) \equiv 1$, we obtain the interpolant $q=\hat f(\omega)$ as 
\begin{equation}
\label{eq:interpolant_linear_in_data}
	\hat f(\omega) = \sum_{j=1}^k \varphi_j(\omega) p_j.
\end{equation}
\begin{remark}
The interpolant in \eqref{eq:interpolant_linear_in_data} is a weighted linear combination of the sample data points
and can thus be constructed in the same fashion in arbitrary vector spaces.
Lagrange interpolation, radial basis function interpolation and Gaussian process regression/Kriging eventually boil down to the form of \eqref{eq:interpolant_linear_in_data} and differ only by the choice of the weight functions $\omega \to \varphi_j(\omega)$.
Hence, all these methods may be considered as examples of weighted barycentric interpolation approaches.
\end{remark}
\paragraph{Parameterized solutions via the implicit function theorem}
Introduce 
\begin{equation*} 
  G\colon\R^m \times \R^d \to \R^m, \quad (q, \omega) \mapsto G(q,\omega) = \grad_q L(q,\omega) = \sum_{j=1}^k \varphi_j(\omega) (q-p_j).
\end{equation*}
Suppose that $G$ has a root at $(q^*, \omega^*)$. The differential of $G$ by $q$ at $(q^*, \omega^*)$ is 
\begin{equation*}
	\Diff_q G(q^*, \omega^*) = \sum_{j=1}^k \varphi_j(\omega^*) I = I,
\end{equation*}
and is invertible. By the implicit function theorem, there exists 
$\hat f\colon \R^d\to \R^m, \omega\mapsto \hat f(\omega)$ such that on a suitably small domain around $\omega^*$, it holds $G(\hat f(\omega),\omega) \equiv 0$.
Implicit differentiation yields (locally) 
\begin{equation*}
	\Diff_q G(\hat f(\omega), \omega) \circ D\hat f(\omega) + \Diff_\omega G(\hat f(\omega), \omega) =0.
\end{equation*}
Note that $\Diff_q G(\hat f(\omega), \omega) = \Diff_q \grad_q L(\hat f(\omega), \omega)=\Hess_q L(\hat f(\omega),\omega)$.
Hence, we obtain the following equation that relates the partial derivatives of $\hat f$, which acts as the interpolant, to those of the weight functions $\varphi_j$. 
\begin{equation}
\label{eq:implderivEuc}
	\Hess_q L(\hat f(\omega),\omega) [\partial_i \hat f(\omega)]
	= -\sum_{j=1}^k {\partial_i \varphi_k(\omega)(\hat f(\omega)-p_j)}
	= \sum_{j=1}^k {\partial_i \varphi_k(\omega)\Log_{\hat f(\omega)}(p_j)}.
\end{equation}
It holds $\Hess_q L(\hat f(\omega),\omega) = I$ so  that
\begin{equation*}
	\partial_i \hat f(\omega) =
	\begin{pmatrix}[ccc]
	\Log_{\hat f(\omega)}(p_1) & \cdots &\Log_{\hat f(\omega)}(p_k) 
	\end{pmatrix}
	\begin{pmatrix}
	  \partial_i \varphi_1(\omega)\\
	  \vdots\\
	  \partial_i \varphi_k(\omega)
	\end{pmatrix}.
\end{equation*}
If the partial derivatives $\partial_i \hat f(\omega_j)$, $i=1,\ldots,d$ are known at the sample sites $\omega_j, j=1,\ldots,k$, then the above equation imposes conditions
on the partial derivatives of the weight functions thereat.
At sample location $\omega_l$, it holds $\hat f(\omega_l) = p_l$ and
\begin{equation}
\label{eq:partialbaryHermiteEuc}
	\partial_i \hat f(\omega_l)
	= -\sum_{j=1, j\neq l}^k {\partial_i \varphi_j(\omega_l)(p_l-p_j)}
	= \sum_{j=1, j\neq l}^k {\partial_i \varphi_j(\omega_l)\Log_{p_l}(p_j)}.
\end{equation}
In order to be able to match any prescribed, sampled partial derivatives, 
$\partial_i \hat f(\omega_l)$, we require the sets
\begin{equation*}
	\mathcal L_l \coloneqq \{p_j - p_l\mid j=1,\ldots,k, j\neq l\} = \{\Log_{p_l}(p_j)\mid j\neq l\}
	\subset T_{p_l}\R^m\simeq \R^m
\end{equation*}
to span the full tangent space $T_{p_l}\R^m\simeq \R^m$.
As a consequence, at least $k\geq m+1$ sample points must be available to make the barycentric Hermite interpolation problem well-defined.
If $k>m+1$, the equation system \eqref{eq:partialbaryHermiteEuc} is underdetermined and infinitely many valid choices of partial derivative values $\partial_i \varphi_j(\omega_l)$ exist that allow to represent the sampled partial derivatives.
\section{Karcher's computation of the gradient and the Hessian of the Riemannian distance function}
\label{app:proofs}
To make this paper self-contained, we recap the proof of Karcher's Theorem, \Cref{thm:grad_dist}.
To this end, we first list two standard facts about covariant derivatives.
\begin{lemma}[Product rule, {\cite[Chapter 2, Prop. 3.2]{DoCarmo:1992}}]
 \label{lem:productrule_codiff_along_c}
 Let $\covDeriv$ denote  the Le\-vi-Ci\-vi\-ta connection along a curve $c:I\to \mcM$. It holds
 \begin{equation*}
  \deriv[t] \langle X(t), Y(t)\rangle_{c(t)}
  = \langle \frac{\mathrm{D}X}{\mathrm{d}t} (t), Y(t)\rangle_{c(t)} +
   \langle X(t), \frac{\mathrm{D}Y}{\mathrm{d}t}(t)\rangle_{c(t)} 
 \end{equation*}
 for all smooth vector fields $X,Y$ along $c$.
\end{lemma}
The next statement corresponds to the Theorem of Schwarz in classical multivariate calculus.
 \begin{lemma}[{\cite[Chapter 3, Lemma 3.4]{DoCarmo:1992}}]
 \label{lem:schwarz4codiff}
  Let $\alpha\colon I\times J \rightarrow \mcM, (s,t)\mapsto \alpha(s,t)$
  be smooth.
  Note that for $s$ fixed, $t\mapsto \alpha(s,t)$ is a smooth manifold curve,
  likewise for $t$ fixed.
  Hence,
  \begin{align*}
   \frac{\partial \alpha}{\partial s} (s,t) &= \deriv[\sigma]\big\vert_{\sigma = 0}\alpha(s+\sigma,t) \in T_{\alpha(s,t)}\mcM,\\
   \frac{\partial \alpha}{\partial t} (s,t) &= \deriv[\tau]\big\vert_{\tau = 0}\alpha(s,t+\tau)  \in T_{\alpha(s,t)}\mcM,\\
  \end{align*}
  and it holds
  \begin{equation*}
   \frac{\mathrm{D}}{\partial t} \frac{\partial\alpha}{\partial s}(s,t) =
   \frac{\mathrm{D}}{\partial s} \frac{\partial\alpha}{\partial t} (s,t).
  \end{equation*}
  Here,
  \begin{equation*}
    \frac{\mathrm{D}}{\partial t} \frac{\partial\alpha}{\partial s}(s,t) \coloneqq 
    \covDeriv[\tau]\big\vert_{\tau = 0} 
    \tau \mapsto \alpha(s,t+\tau).
  \end{equation*}
  Likewise for $\frac{\mathrm{D}}{\partial s} \frac{\partial\alpha}{\partial t} (s,t)$.
 \end{lemma}
\begin{theorem}[{\cite[Thm. 1.2]{Karcher:1977}}]
\label{thm:Karcher1977}
Let $\mcM$ be a complete Riemannian manifold and let $p\in \mcM$ be a fixed point.
Let $B_\rho(p)$ be a geodesic ball around $p_*$ such that the geodesics between any two points inside $B_\rho(p)$ are unique and minimizing.
Define 
\begin{equation*}
	f\colon B_\rho(p)\to \R, \quad q \mapsto \frac12 \dist(q,p)^2.
\end{equation*}
Then
\begin{equation*}
	\grad_q f = -\Log_q(p).
\end{equation*}
\end{theorem}
\begin{proof}
	Let $\gamma\colon [0,1]\to B_\rho(p)$ be a geodesic with $\gamma(0) = q, \dot \gamma(0) = v\in T_q\mcM$.
	Consider a variation of geodesics 
	\begin{equation*}
		c_p(s,t) = \Exp_p(s\Log_p(\gamma(t)))
	\end{equation*}
as illustrated in Figure \ref{fig:GeoVar}.
\begin{figure}[tbp]
    \centering
    \includegraphics[width=0.5\textwidth]{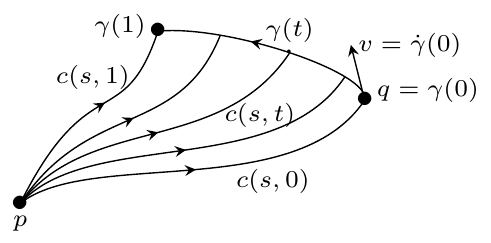}
    \caption{Illustration of the variation through geodesics considered in the proof of Theorem \ref{thm:Karcher1977}.
    }
    \label{fig:GeoVar}
\end{figure}
	As in \cite{Karcher:1977}, we write 
	\begin{equation*}
		c_p'(s,t) = \partial_s c_p(s,t), \quad \dot c_p(s,t) = \partial_t c_p(s,t)
	\end{equation*}
	and observe that 
	\begin{enumerate}[(i)]
	\item $c_p(0,t) = \Exp_p(0) = p$. (All geodesics $s\mapsto c_p(s,t)$ start from $p$).
	\item $c_p(1,t) = \gamma(t)$. (All geodesics $s\mapsto c_p(s,t)$ end at $\gamma(t)$).
		 Moreover,  $\dot \gamma(t) = \dot c_p(1,t)$.
	\item the terminal velocity of the geodesic $s\mapsto c_p(s,t)$ matches minus the velocity of the geodesic
	emanating from $\gamma(t) = c_p(1,t)$ to reach $p$. Hence,
	\begin{equation*}
		c_p'(1,t) = -\Log_{\gamma(t)}(p).
	\end{equation*}
	\item all geodesics are constant-speed curves so that
	\begin{equation*}
		\dist(p,\gamma(t)) =\lVert \Log_p(\gamma(t)\rVert_p =  \lVert c_p'(0,t)\rVert_p = \lVert c_p'(s,t)\rVert_{c_p(s,t)} \text{ for all } s\in[0,1].
	\end{equation*}
	Because of this, we can write 
	\begin{equation*}
		\lVert c'_p(s,t)\rVert^2 = \int_0^1\lVert c'_p(s,t)\rVert^2 \dInt s.
	\end{equation*}
	\item 	When taking covariant derivatives, we may change the order of differentiation according to \Cref{lem:schwarz4codiff}
	\begin{equation*}
		\frac{\mathrm{D}}{\partial_s}\dot c_p(s,t) = \frac{\mathrm{D}}{\partial_t} c_p'(s,t).
	\end{equation*}
	\end{enumerate}
	We calculate for any $v\in T_q\mcM$ the differential as
	\begin{align*}
		\diff{}f_q(v)
		&= \deriv[t]\big\vert_{t=0} f(\gamma(t)) = \frac12 \deriv[t]\big\vert_{t=0}  \dist(\gamma(t),p)^2\\
		&= \frac12 \deriv[t]\big\vert_{t=0}  \int_0^1\lVert c'_p(s,t)\rVert^2 \dInt s\\
		&= \frac12  \int_0^1 \deriv[t]\big\vert_{t=0}  \langle c'_p(s,t), c'_p(s,t)\rangle\dInt s\\
		&= \int_0^1   \langle \frac{\mathrm{D}}{\partial_t} c'_p(s,t), c'_p(s,t)\rangle\big\vert_{t=0}\dInt s\\
		&= \int_0^1 \langle \frac{\mathrm{D}}{\partial_s} \dot c_p(s,t), c'_p(s,t)\rangle\big\vert_{t=0}\dInt s\\
		&= \int_0^1 \left(\langle \frac{\mathrm{D}}{\partial_s} \dot c_p(s,t), c'_p(s,t)\rangle
        + \langle\dot c_p(s,t), \frac{\mathrm{D}}{\partial_s} c'_p(s,t)\rangle \right)\big\vert_{t=0}\dInt s,\\
	\end{align*}
    where the second term vanishes, since $c_p(s,t)$ is a geodesic in $s$ for any fixed $t$. We continue
	\begin{align}
	    \nonumber
		\diff{}f_q(v)
		&= \int_0^1 \deriv[s]  \langle \dot c_p(s,t), c'_p(s,t)\rangle\dInt s\big\vert_{t=0} \\   
	    \nonumber
		&= \left[  \langle \dot c_p(s,t), c'_p(s,t)\rangle   \right]_0^1  \big\vert_{t=0}, \\
		&=   \langle \dot c_p(1,t), c'_p(1,t)\rangle  +  \langle \dot c_p(0,t), c'_p(0,t)\rangle  \big\vert_{t=0},
	\label{eq:var1stOrder_deriv}
	\end{align}
    where again the second summand is zero since $c_p(0,t) = p$ is constant. We finally obtain
	\begin{equation*}
		\diff{}f_q(v)
		=   \langle \dot \gamma(t), -\Log_{\gamma(t)}(p)\rangle \big\vert_{t=0}\\
	\nonumber
		=  \langle v,  -\Log_{q}(p)\rangle
	\end{equation*}
	and hence $\grad_q f = -\Log_{q}(p)\in T_q\mcM$.
\end{proof}
\paragraph{The Hessian of the Riemannian distance function.}
Let $g\colon\mcM \to \R$ be a smooth scalar function. 
Following \cite[Section 6A]{Kuhnel:2015}, it may be considered as a $(0,0)$-tensor.
The covariant derivative of $g$ (in the sense of \cite[Def. 6.2]{Kuhnel:2015})
is a $(0,1)$-tensor $(\nabla g) (X)$ that maps a vector field $X$ to a scalar function.
More precisely, we have
\begin{equation*}
 \nabla g\colon \mathcal{X}(\mcM) \to C^\infty(\mcM),
 X \mapsto \left(p \mapsto (\nabla f(X))(p) = \diff{}g_p(X(p)) = \langle \grad f(p), X(p) \rangle_p\right).
\end{equation*}
The second covariant derivative, $\nabla^2 g$, called the \emph{Hessian} of $g$, 
is obtained by taking the covariant derivative of the $(0,1)$-tensor $(\nabla g)$.
This produces a $(0,2)$-tensor.
A calculation shows 
\begin{equation}
\nonumber
  \nabla^2 f (X,Y) = \langle   \nabla_X \grad f, Y \rangle.
\end{equation}
The Hessian $(1,1)$-tensor associated with $g$ is thus
\begin{equation*}
    \Hess f\colon \mathcal{X}(\mcM) \to \mathcal{X}(\mcM) ,  \quad X \mapsto \nabla_X \grad g.
\end{equation*}

The next lemma provides a path to compute the Hessian operator at a point $p$ and tangent vector $v=X(p)\in T_p\mcM$.
\begin{lemma}[cf. {\cite[Prop. 5.5.4]{AbsilMahonySepulchre:2008}}]
	Let $g\colon\mcM\to \R$ be a smooth scalar function on a 
	Riemannian manifold $\mcM$.
	Let $\gamma$ be a geodesic with $\gamma(0)= q$, $\dot\gamma(0)=v$. 
	Then
	\begin{equation*}
		\langle \Hess g(q)[v]\hspace{0.1cm},\hspace{0.1cm} v\rangle = \frac{\mathrm{d}^2}{\mathrm{d}t^2}\big\vert_{t=0} (g\circ \gamma)(t).
	\end{equation*}
\end{lemma}
\begin{proof}
It holds 
\begin{equation*}
	 \deriv[t] (f\circ \gamma)(t) = \diff{}f_{\gamma(t)}(\dot \gamma(t)) = \langle \grad f(\gamma(t)), \dot \gamma(t) \rangle.
\end{equation*}
Taking the second derivative, we obtain by the rules of covariant differentiation along a curve and 
the product rule of Lemma \ref{lem:productrule_codiff_along_c} that
\begin{align*}
	 \frac{\mathrm{d}^2}{\mathrm{d}t^2} (f\circ \gamma)(t) 
	 &=  \deriv[t] \langle \grad f(\gamma(t)), \dot \gamma(t) \rangle\\
	 &=  \langle  \covDeriv[t] \grad f(\gamma(t)), \dot \gamma(t) \rangle +  \langle \grad f(\gamma(t)),   \covDeriv[t] \dot \gamma(t)\rangle\\
	 &=  \langle  \nabla_{\dot \gamma(t)} \grad f, \dot \gamma(t) \rangle
    \\
	 &= \langle \Hess f(\gamma(t))[\dot\gamma(t)]\hspace{0.1cm},\hspace{0.1cm} \dot\gamma(t)\rangle.
\end{align*}
The third identity holds, because $\covDeriv[t] \dot \gamma(t) = 0$, since $\gamma$ is a geodesic.
At $t=0$, this yields 
\begin{equation} 
\label{eq:Hess_vv}
	\frac{\mathrm{d}^2}{\mathrm{d}t^2} (f\circ \gamma)(0)
	= \langle \Hess f(q)[v]\hspace{0.1cm},\hspace{0.1cm} v\rangle,
\end{equation}
which finishes the proof.
\end{proof}
For computing the Hessian of the distance function $q\mapsto \frac12 \dist(p,q)^2$, we utilize the same variation of geodesics 
$c_p(s,t)$ as in the above proof of Theorem \ref{thm:Karcher1977} 
and continue the calculation from \eqref{eq:var1stOrder_deriv}.
Using that $c_p(1,t)$ is a geodesic, we obtain
\begin{align*}
	 \frac{\mathrm{d}^2}{\mathrm{d}t^2} \frac12 \dist(\gamma(t), p)^2
	 &=  \deriv[t]  \left(\deriv[t]\frac12 \dist(\gamma(t), p)^2\right)
	 \stackrel{\eqref{eq:var1stOrder_deriv}}{=}\deriv[t] \langle \dot c_p(1,t), c_p'(1,t) \rangle\\
	 &=  \langle \covDeriv[t]\dot c_p(1,t), c_p'(1,t) \rangle +  \langle\dot  c_p(1,t) ,   \covDeriv[t]c_p'(1,t) \rangle\\
	 &=  \langle \dot \gamma(t) ,   \covDeriv[s] \dot c_p(1,t) \rangle.
\end{align*}
At $t=0$, we obtain for $\gamma(0) = q$, $\dot \gamma(0) = v$,
\begin{equation}
\label{eq:Hessian_via_var_of_geos}
 \langle \Hess f(q)[v]\hspace{0.1cm},\hspace{0.1cm} v\rangle =
	\frac{\mathrm{d}^2}{\mathrm{d}t^2}\big\vert_{t=0} \dist(\gamma(t), p)  = 
	 \langle \covDeriv[s] \dot c_p(1,0), v \rangle.
\end{equation}
\textbf{Remark}
	Associated with the variation $c_p(s,t) = \Exp_p(s\Log_p(\gamma(t)))$ is the variation field
	\begin{equation*}
   s\mapsto  \dot c_p(s,0) = 	\partial_t\big\vert_{t=0} c_p(s,t)  = \diff{}(\Exp_p)_{(s\Log_p(q)}\left(s\cdot \diff{}(\Log_p)_{q}(v)\right)
	\end{equation*}
	For notational convenience, define $x = \Log_p(q), y = \diff{}(\Log_p)_{q}(v) \in T_p\mcM$.
	The above variation vector field coincides with
	\begin{equation*}
		s\mapsto  \dot c_p(s,0) = 	\partial_t\big\vert_{t=0} c_p(s,t)  =
		\partial_\tau\big\vert_{\tau = 0} \Exp_p(s \cdot (x + \tau y)).
	\end{equation*}
	This is a variation of the geodesic $s\mapsto \hat c_p(s) \coloneqq c_p(s,0)$ from $p$ to $q$ through geodesics, thus a Jacobi field.\footnote{
	More precisely, the unique Jacobi field along the geodesic $c_p(s,0)$ that matches the initial values
	$J(0)  = \partial_\tau\big\vert_{\tau =0} \Exp_p(0\cdot(x+\tau y)) = \partial_\tau\big\vert_{\tau =0} p = 0$
	and
	\begin{align*}
	J'(0)&= \covDeriv[s]\big\vert_{s =0} \partial_\tau\big\vert_{\tau =0} Exp_p(s \cdot (x + \tau y))  
	=  \covDeriv[\tau]\big\vert_{\tau =0} \partial_s\big\vert_{s =0} Exp_p(s \cdot (x + \tau y)) \\
	&= \covDeriv[\tau]\big\vert_{\tau =0}  (x + \tau y) = y = \diff{}(\Log_p)_{q}(v).
	\end{align*}
	}
	Hence, with $s\mapsto J(s) = \partial_\tau\big\vert_{\tau = 0} \Exp_p(s \cdot (x + \tau y)) $, we have
	\begin{equation*}
		J(s)  = \dot c_p(s,0), \quad J'(s) = \covDeriv[s]J(s) =  \covDeriv[s]\dot c_p(s,0).
	\end{equation*}
	Thus, we may write
	\begin{align*}
 \langle \Hess f(q)[v]\hspace{0.1cm},\hspace{0.1cm} v\rangle &=
	 \langle J'(1), v \rangle.
    \end{align*}

\printbibliography

@book{AbsilMahonySepulchre:2008,
    AUTHOR    = {Absil, P.-A. and Mahony, R. and Sepulchre, R.},
    PUBLISHER = {Princeton University Press},
    YEAR      = {2008},
    DOI       = {10.1515/9781400830244},
    TITLE     = {{O}ptimization {A}lgorithms on {M}atrix {M}anifolds}
}

@book{DoCarmo:1992,
    AUTHOR    = {do Carmo, Manfredo Perdig\~{a}o},
    PUBLISHER = {Birkh\"{a}user Boston, Inc., Boston, MA},
    YEAR      = {1992},
    ISBN      = {0-8176-3490-8},
    SERIES    = {Mathematics: Theory \& Applications},
    TITLE     = {{R}iemannian {G}eometry}
}

@book{GolubVanLoan:1996,
author = {Golub, G~.H. and Van Loan, C.~F.},
title  = {Matrix Computations},
publisher = {The John Hopkins University Press},
year = {1996},
edition = {3rd},
address = {Baltimore -- London},
}

@book{Kuhnel:2015,
  title={Differential Geometry: Curves -- Surfaces --  Manifolds},
  author={K{\"u}hnel, W.},
  ISBN={9781470423209},
  lccn={2015018451},
  series={Student Mathematical Library},
  year={2015},
  publisher={American Mathematical Society},
  edition={3}
}

@book{Lee:2012,
  title={Introduction to Smooth Manifolds},
  author={Lee, J.~M.},
  isbn={9781441999825},
  lccn={2012945172},
  series={Graduate Texts in Mathematics},
  year={2012},
  publisher={Springer New York}
}

@book{Lee:1997,
  AUTHOR = {Lee, John M.},
  PUBLISHER = {Springer-Verlag, New York},
  DATE = {1997},
  DOI = {10.1007/b98852},
  ISBN = {0-387-98271-X},
  NOTE = {An introduction to curvature},
  SERIES = {Graduate Texts in Mathematics},
  TITLE = {Riemannian Manifolds},
  VOLUME = {176},
}

@book{Higham:2008,
  author =       "Higham, N.~J.",
  title =        "Functions of Matrices: {Theory} and Computation",
  publisher =    "Society for Industrial and Applied Mathematics",
  address =      "Philadelphia, PA, USA",
  year =         "2008",
  pages =        "xx+425",
  isbn =         "978-0-898716-46-7"
}

@article{Karcher:1977,
  AUTHOR = {Karcher, H.},
  DATE = {1977},
  DOI = {10.1002/cpa.3160300502},
  JOURNALTITLE = {Communications on Pure and Applied Mathematics},
  NUMBER = {5},
  PAGES = {509--541},
  TITLE = {Riemannian center of mass and mollifier smoothing},
  VOLUME = {30},
}

@article{AfsariTronVidal:2013,
  AUTHOR = {Afsari, Bijan and Tron, Roberto and Vidal, René},
  DATE = {2013},
  DOI = {10.1137/12086282X},
  ISSN = {0363-0129},
  JOURNALTITLE = {SIAM Journal on Control and Optimization},
  NUMBER = {3},
  PAGES = {2230--2260},
  TITLE = {On the convergence of gradient descent for finding the Riemannian center of mass},
  VOLUME = {51},
}

@Article{Sander:2016,
  AUTHOR = {Sander, Oliver},
  DATE = {2016},
  DOI = {10.1093/imanum/drv016},
  ISSUE = {1},
  JOURNALTITLE = {IMA Journal of Numerical Analysis},
  PAGES = {238--266},
  TITLE = {Geodesic finite elements of higher order},
  VOLUME = {36},
}

@article{Grohs:2015,
  AUTHOR = {Grohs, Philipp and Hardering, Hanne and Sander, Oliver},
  DATE = {2015},
  DOI = {10.1007/s10208-014-9230-z},
  JOURNALTITLE = {Foundations of Computational Mathematics},
  NUMBER = {6},
  PAGES = {1357--1411},
  TITLE = {Optimal a priori discretization error bounds for geodesic finite elements},
  VOLUME = {15},
}

@article{Sander:2012,
    author={Sander,Oliver},
    year={2012},
    title={Geodesic finite elements on simplicial grids},
    journal={International journal for numerical methods in engineering},
    volume={92},
    number={12},
    pages={999-1025},
}

@book{Buhmann:2003,
    author = {Buhmann, M.~D.}, 
    title =  {Radial Basis Functions},
    series = {Cambridge Monographs on Applied and Computational Mathematics},
    volume = {12},
    publisher =  {Cambridge University Press},
    year = {2003},
    address = {Cambridge, UK},
}

@article{Seguin:2022,
    author = {S\'{e}guin, Axel and Kressner, Daniel},
    title = {Continuation Methods for Riemannian Optimization},
    journal = {SIAM Journal on Optimization},
    volume = {32},
    number = {2},
    pages = {1069-1093},
    year = {2022},
    doi = {10.1137/21M1428650},
    eprint = {2106.08839},
    eprinttype={arxiv}
}

@book{ForresterSobesterKeane:2008,
author = {Forrester, A.~I.~J. and Sobester, A. and Keane, A.~J.},
title  = {Engineering Design via Surrogate Modelling: A Practical Guide},
publisher = {John Wiley \& Sons},
address = {United Kingdom},
year = {2008},
}

@article{Zimmermann:2013,
author = {Zimmermann, R.},
title = {On the Maximum Likelihood Training of Gradient-Enhanced Spatial {G}aussian Processes},
journal = {SIAM Journal on Scientific Computing},
volume = {35},
number = {6},
pages = {A2554-A2574},
year = {2013},
doi = {10.1137/13092229X},
}

@article{Zimmermann:2020,
author = {Zimmermann, R.},
title = {{H}ermite Interpolation and Data Processing Errors on {R}iemannian Matrix Manifolds},
journal = {SIAM Journal on Scientific Computing},
volume = {42},
number = {5},
pages = {A2593-A2619},
year = {2020},
doi = {10.1137/19M1282878},
}

@article{Jakubiak:2006,
author = {Jakubiak, J. and Leite, F.~S. and Rodrigues, R.},
year = {2006},
pages = {177-191},
title = {A two-step algorithm of smooth spline generation on {R}iemannian manifolds},
volume = {194},
journal = {Journal of Computational and Applied Mathematics},
doi = {10.1016/j.cam.2005.07.003},
}

@article{Noakes:2007,
author = "Popiel, T. and Noakes, L.",
title = "B{\'e}zier curves and {C2} interpolation in {R}iemannian manifolds",
journal = "Journal of Approximation Theory",
volume = {148},
number = {2},
year = {2007},
pages = {111--127},
issn = {0021-9045},
}

@article{Polthier:2013,
author = {E. Nava-Yazdani and K. Polthier},
title = {{D}e {C}asteljau's algorithm on manifolds},
journal = {Computer Aided Geometric Design},
volume = {30},
number = {7},
year = {2013},
pages = {722--732},
doi = {10.1016/j.cagd.2013.06.002},
}

@article{AbsGouseWirth:2016,
journal = {SIAM Journal on Imaging Sciences},
year = {2016}, 
volume = {9},
number = {4},
pages = {1788--1828},
title = {Differentiable Piecewise-{B}{\'e}zier Surfaces on {R}iemannian Manifolds},
author = {P.-A. Absil and P.-Y. Gousenbourger and P. Striewski and B. Wirth},
url = {10.1137/16M1057978},
}

@article{SAMIR:2019,
title = {C1 interpolating {B}{\'e}zier path on {R}iemannian manifolds, with applications to {3D} shape space},
journal = {Applied Mathematics and Computation},
volume = {348},
pages = {371 -- 384},
year = {2019},
doi = {10.1016/j.amc.2018.11.060},
author = "Samir, C. and Adouani, I.",
}

@phdthesis{Amsallem:2010,
author = {Amsallem, D.},
title  = {Interpolation on Manifolds of CFD-based Fluid and Finite Element-based Structural Reduced-order Models for On-line Aeroelastic Prediction},
school = {Stanford University}, 
year   = {2010},
}

@article{GouseMassartAbsil:2018,
author = {Gousenbourger, P.-Y. and Massart, E. and Absil, P.-A.},
title = {Data Fitting on Manifolds with Composite {B}{\'e}zier-Like Curves and Blended Cubic Splines},
journal = {Journal of Mathematical Imaging and Vision},
year = {2018},
volume = {online},
doi = {10.1007/s10851-018-0865-2},
pages = {1--27},
}

@article{Wallner:2005,
title = {Convergence and C1 analysis of subdivision schemes on manifolds by proximity},
journal = {Computer Aided Geometric Design},
volume = {22},
number = {7},
pages = {593 -- 622},
year = {2005},
note = {Geometric Modelling and Differential Geometry},
doi = {10.1016/j.cagd.2005.06.003},
author = {J. Wallner and N. Dyn},
}

@article{Dyn:2017,
title = {Manifold-valued subdivision schemes based on geodesic inductive averaging},
journal = {Journal of Computational and Applied Mathematics},
volume = {311},
pages = {54-67},
year = {2017},
issn = {0377-0427},
doi = {10.1016/j.cam.2016.07.008},
author = {Nira Dyn and Nir Sharon},
}

@online{BenZion:2022,
  eprint = {2203.02903},
  eprinttype = {arxiv},
  author = {Vardi, Hofit Ben-Zion and Dyn, Nira and Sharon, Nir},
  title = {Geometric Hermite Interpolation in $\mathbb{R}^n$ by Refinements},
  year = {2022},
}

@Article{Narcowich:1995,
author="Narcowich, F.",
title="Generalized {H}ermite Interpolation and Positive Definite Kernels on a {R}iemannian Manifold",
journal="Journal of Mathematical Analysis and Applications",
year="1995",
volume="190",
pages="165--193",
}

@article{Allasia:2018,
title = {{H}ermite–{B}irkhoff interpolation on scattered data on the sphere and other manifolds},
journal = {Applied Mathematics and Computation},
volume = {318},
pages = {35-50},
year = {2018},
note = {Recent Trends in Numerical Computations: Theory and Algorithms},
doi = {10.1016/j.amc.2017.05.018},
author = {Allasia, G. and Roberto, C. and {De Rossi}, A.},
}

@article{BergmannGousenbourger:2018,
	author = {Bergmann, Ronny and Gousenbourger, Pierre-Yves},
	date = {2018},
	doi = {10.3389/fams.2018.00059},
	eprint = {1807.10090},
	eprinttype = {arXiv},
	journaltitle = {Frontiers in Applied Mathematics and Statistics},
	title = {A variational model for data fitting on manifolds by minimizing the acceleration of a B{\'e}zier curve},
}

@incollection{Zimmermann:2021,
	author="Zimmermann, R.",
	editor={Benner, P. and Grivet-Talocia, S. and Quarteroni, A. and Rozza, G. and Schilders, W. and Silveira, L.~M.},
	title="Manifold interpolation",
	bookTitle="System- and Data-Driven Methods and Algorithms",
	volume = {1},
	series = "Model Order Reduction",
	year="2021",
	publisher = {De Gruyter},
	address="Boston",
	pages = "229--274",
}

\end{document}